\documentclass[UTF-8,reqno]{amsart}
\usepackage{enumerate}
\usepackage{amssymb,url,color, booktabs}
\usepackage{mathrsfs}

\numberwithin{equation}{section}

\newcommand{\be}{\begin{eqnarray}}
\newcommand{\ee}{\end{eqnarray}}
\newcommand{\ce}{\begin{eqnarray*}}
\newcommand{\de}{\end{eqnarray*}}
\newtheorem{theorem}{Theorem}[section]
\newtheorem{lemma}[theorem]{Lemma}
\newtheorem{remark}[theorem]{Remark}
\newtheorem{definition}[theorem]{Definition}
\newtheorem{proposition}[theorem]{Proposition}
\newtheorem{Examples}[theorem]{Example}
\newtheorem{corollary}[theorem]{Corollary}

\def\nor{|\mspace{-3mu}|\mspace{-3mu}|}

\def\eps{\varepsilon}

\def\e{\mathrm{e}}

\def\p{\partial}

\def\[{{\Big[}}
\def\]{{\Big]}}
\def\<{{\langle}}
\def\>{{\rangle}}
\def\({{\Big(}}
\def\){{\Big)}}

\def\bx{{\mathbf{x}}}

\def\dif{{\mathord{{\rm d}}}}

\def\no{\nonumber}
\def\={&\!\!=\!\!&}

\def\cB{{\mathcal B}}

\def\cF{{\mathcal F}}

\def\cM{{\mathcal M}}

\def\cP{{\mathcal P}}

\def\cW{{\mathcal W}}

\def\mC{{\mathbb C}}

\def\mE{{\mathbb E}}

\def\mH{{\mathbb H}}
\def\mI{{\mathbb I}}

\def\mK{{\mathbb K}}
\def\mL{{\mathbb L}}
\def\mM{{\mathbb M}}
\def\mN{{\mathbb N}}

\def\mP{{\mathbb P}}
\def\mQ{{\mathbb Q}}
\def\mR{{\mathbb R}}

\def\mZ{{\mathbb Z}}

\def\bP{{\mathbf P}}

\def\bE{{\mathbf E}}
\def\1{{\mathbf{1}}}

\def\sE{{\mathscr E}}
\def\sF{{\mathscr F}}

\def\sI{{\mathscr I}}

\def\sL{{\mathscr L}}
\def\sM{{\mathscr M}}

\def\geq{\geqslant}
\def\leq{\leqslant}

\def\div{\mathord{{\rm div}}}

\def\eps{\varepsilon}

\def\e{\mathrm{e}}

\def\p{\partial}

\def\[{{\Big[}}
\def\]{{\Big]}}
\def\<{{\langle}}
\def\>{{\rangle}}
\def\({{\Big(}}
\def\){{\Big)}}

\def\bx{{\mathbf{x}}}

\def\dif{{\mathord{{\rm d}}}}

\def\no{\nonumber}
\def\={&\!\!=\!\!&}
\def\bt{\begin{theorem}}
\def\et{\end{theorem}}
\def\bl{\begin{lemma}}
\def\el{\end{lemma}}
\def\br{\begin{remark}}
\def\er{\end{remark}}
\def\bx{\begin{Examples}}
\def\ex{\end{Examples}}
\def\bd{\begin{definition}}
\def\ed{\end{definition}}
\def\bp{\begin{proposition}}
\def\ep{\end{proposition}}
\def\bc{\begin{corollary}}
\def\ec{\end{corollary}}

\def\geq{\geqslant}
\def\leq{\leqslant}

\def\div{\mathord{{\rm div}}}

\def\bK{{\mathbf K}}

\def\bP{{\mathbf P}}
\def\bS{{\mathbf S}}

\def\<{\langle} \def\>{\rangle}

 \def\beq{\begin{equation}}  
 
\def\e{\text{\rm{e}}}

\allowdisplaybreaks

\begin{document}

\title{Well-posedness of distribution dependent SDEs with singular drifts}
\date{}
\author{Michael R\"ockner and Xicheng Zhang}

\address{Michael R\"ockner:
Fakult\"at f\"ur Mathematik, Universit\"at Bielefeld,
33615, Bielefeld, Germany\\
Email: roeckner@math.uni-bielefeld.de
 }

\address{Xicheng Zhang:
School of Mathematics and Statistics, Wuhan University,
Wuhan, Hubei 430072, P.R.China\\
Email: XichengZhang@gmail.com
 }

\thanks{
This work is supported by NNSFC grant of China (No. 11731009) and the DFG through the CRC 1283 
``Taming uncertainty and profiting from randomness and low regularity in analysis, stochastics and their applications''.
}

\begin{abstract}
Consider the following distribution dependent SDE:
$$
{\mathrm d} X_t=\sigma_t(X_t,\mu_{X_t}){\mathrm d} W_t+b_t(X_t,\mu_{X_t}){\mathrm d} t,
$$
where $\mu_{X_t}$ stands for the distribution of $X_t$. In this paper for non-degenerate $\sigma$, 
we show the strong well-posedness of the above SDE under some 
integrability assumptions in the spatial variable and Lipschitz continuity in $\mu$ about $b$ and $\sigma$.
In particular, we extend the results of Krylov-R\"ockner \cite{Kr-Ro} to the distribution dependent case.


\end{abstract}

\maketitle \rm

\section{Introduction}

Let $\cP(\mR^d)$ be the space of all probability measures over $(\mR^d,\cB(\mR^d))$, which is endowed with the weak convergence topology.
Consider the following distribution dependent stochastic differential equation (abbreviated as DDSDEs):
\begin{align}\label{SDE}
\dif X_t=b_t(X_t,\mu_{X_t})\dif t+\sigma_t(X_t,\mu_{X_t})\dif W_t,
\end{align}
where $b:\mR_+\times\mR^d\times\cP(\mR^d)\to\mR^d$
and $\sigma:\mR_+\times\mR^d\times\cP(\mR^d)\to\mR^d\otimes\mR^d$ are two Borel measurable functions,
$W$ is a $d$-dimensional standard Brownian motion on some filtered probability space
$(\Omega,\sF,(\sF_t)_{t\geq 0}, \bP)$, and $\mu_{X_t}:=\bP\circ X^{-1}_t$ is the time marginal of $X_t$ at time $t$,
By It\^o's formula, it is easy to see that $\mu_{X_t}$ satisfies the following non-linear Fokker-Planck equation (abbreviated as FPE) in the distributional sense:
\begin{align}\label{Non}
\p_t\mu_{X_t}=(\sL_t^{\sigma^X})^*\mu_{X_t}+\div(b^X_t\mu_{X_t}),
\end{align}
where $\sigma^X_t(x):=\sigma_t(x,\mu_{X_t})$, $b^X_t(x):=b_t(x,\mu_{X_t})$, and $(\sL_t^{\sigma^X})^*$ is the adjoint operator
of the following second order partial differential operator
\begin{align}\label{ET2}
\sL^{\sigma^X}_t f(x):=\frac{1}{2}\sum_{i,j,k=1}^d(\sigma^{ik}_t\sigma^{jk}_t)(x,\mu_{X_t})\p_i\p_jf(x).
\end{align}
We note that if 
$$
\sigma^X_t(x)=\int_{\mR^d}\sigma_t(x,y)\mu_{X_t}(\dif y),\ \ b^X_t(x)=\int_{\mR^d}b_t(x,y)\mu_{X_t}(\dif y),
$$
then DDSDE \eqref{SDE} is also called mean-field SDE or McKean-Vlasov SDE in the literature,
which naturally appears in the studies of interacting particle systems and mean-field games (see \cite{Ka,Mc,Sz,Ca-De, Ca-Gv-Pa-Sc}, in particular, \cite{Ca-De1}
 and references therein).

\medskip

Up to now, there are numerous papers devoted to the study of this type of nonlinear FPEs and DDSDE \eqref{SDE}.
In \cite{Fu}, Funaki showed the existence of martingale solutions for \eqref{SDE} under broad conditions of Lyapunov's type 
and also the uniqueness under global Lipschitz assumptions. His method is based on a suitable time discretization. Thus, the well-posedness of FPE \eqref{Non}
is also obtained. More recently, under some one-side Lipschitz assumptions, Wang \cite{Wa} showed the strong well-posedness and some functional inequalities
to DDSDE \eqref{SDE}. In \cite{Ha-Si-Sz}, Hammersley, Sitsa and Szpruch proved the existence of weak solutions to SDE \eqref{SDE} on a domain $D\subset\mR^d$
with continuous and unbounded coefficients under Lyapunov-type conditions. Moreover, uniqueness is also obtained under some functional Lyapunov conditions.
Notice that all the above results require  the continuity of coefficients. In \cite{Ch}, Chiang obtained the existence of weak solutions for
time-independent SDE \eqref{SDE} with drifts that have some discontinuities.
When the diffusion matrix is uniformly non-degenerate and $b,\sigma$ are only measurable and of at most linear growth, 
by using the classical Krylov estimates, Mishura and Veretennikov \cite{Mi-Ve} showed the existence of weak solutions. The uniqueness is also proved when $\sigma$
does not depend on $\mu$ and is Lipschitz continuous in $x$ and $b$ is Lipschitz continuous with respect to $\mu$ with Lipschitz constant linearly depending on $x$.
It should be noted that by Schauder's fixed point theorem and Girsanov's theorem, Li and Min \cite{Li-Mi} also obtained the existence and uniqueness of weak solutions 
when $b$ is bounded measurable and $\sigma$ is nondegenerate and Lipschitz continuous.
On the other hand, by a purely analytic argument, Manita and Shaposhnikov \cite{Ma-Sh}
and Manita, Romanov and Shaposhnikov \cite{Ma-Ro-Sh} showed the existence and uniqueness of solutions to the nonlinear FPE \eqref{Non} under quite general assumptions.
As observed in \cite{Ba-Ro}, 
by a result of Trevisan \cite{Tr} (see Theorem \ref{Th23} below), one in fact can obtain the well-posedness of DDSDE \eqref{SDE} from \cite{Ma-Sh} and \cite{Ma-Ro-Sh}.
In \cite{Ba-Ro}, a technique is developed to prove weak existence of solutions to \eqref{SDE} by first solving \eqref{Non} which works also for coefficients whose dependence
on $\mu_{X_t}$ is of ``Nemytskii-type'', i.e., are not continuous in $\mu_{X_t}$ in the weak topology.

\medskip

In this work we are interested in extending Krylov-R\"ockner's result \cite{Kr-Ro}
to the singular distribution dependent case, that is not covered by all of the above results. 
More precisely, we want to show the well-posedness of the following DDSDE:
\begin{align}\label{SDE1}
\dif X_t=\left(\int_{\mR^d}b_t(X_t,y)\mu_{X_t}(\dif y)\right)\dif t+\sqrt{2}\dif W_t,
\end{align}
where $b:\mR_+\times\mR^d\times\mR^d\to\mR^d$ is a Borel measurable function and satisfies
\begin{enumerate}[{\bf (H$^b$)}]
\item $|b_t(x,y)|\leq h_t(x-y)$ for some $h\in L^q_{loc}(\mR_+; \widetilde L^p(\mR^d))$, where $p,q\in(2,\infty)$ satisfy $\frac{d}{p}+\frac{2}{q}<1$, 
and $\widetilde L^p(\mR^d)$ is the localized $L^p$-space defined by \eqref{Ck} below.
\end{enumerate}
Here the advantage of using the localized space $\widetilde L^p(\mR^d)$ is that for any $1\leq p\leq p'\leq\infty$,
$$
L^\infty(\mR^d)+L^{p'}(\mR^d)\subset\widetilde L^{p'}(\mR^d)\subset \widetilde L^p(\mR^d)\subset_{p>d}\mK_{d-1},
$$
where $\mK_{d-1}$ is the usual Kato's class defined by
$$
\mK_{d-1}:=\left\{f:\lim_{\eps\to 0}\sup_{x\in\mR^d}\int_{|x-y|\leq\eps}|x-y|^{1-d}f(y)\dif y=0\right\}.
$$
We note that the above DDSDE is not covered by Huang and Wang's recent results \cite{Hu-Wa} 
since $\mu\mapsto\int_{\mR^d}b_t(x,y)\mu(\dif y)$ is not weakly continuous. In fact, if we let 
\begin{align}\label{BB}
B_t(x,\mu):=\int_{\mR^d}b_t(x,y)\mu(\dif y),\ \mu\in\cP(\mR^d),
\end{align}
then by $|b_t(x,y)|\leq h_t(x-y)$, we only have
\begin{align}\label{TTV}
\nor B_t(\cdot,\mu)-B_t(\cdot,\mu')\nor_p\leq\nor h_t\nor_p\|\mu-\mu'\|_{TV},
\end{align}
where $\|\cdot\|_{TV}$ is the total variation distance, and $\nor\cdot\nor_p$ is defined by \eqref{Ck} below.

\medskip

Throughout this paper we assume $d\geq 2$. One of the main results of this paper is stated as follows (but see also section 4 for corresponding results when 
the diffusion matrix $\sigma$ is non-degenerate, but not constant):
\bt\label{Th11}
Under {\bf (H$^b$)}, for any $\beta>2$ and  initial random variable $X_0$ with finite $\beta$-order moment, there is a unique strong solution to SDE \eqref{SDE1}.
Moreover, the following assertions hold:
\begin{enumerate}[{\rm (i)}]
\item The time marginal law $\mu_t$ of $X_t$ uniquely solves the following nonlinear FPE in the distributional sense:
\begin{align}\label{FK}
\p_t\mu_t=\Delta\mu_t+\div\left(\mu_t(b_t(x,\cdot))\mu_t\right),\ \ 
\lim_{t\downarrow 0}\mu_t(\dif y)=\bP\circ X^{-1}_0(\dif y)
\end{align}
in the class of all measures such that $t\mapsto \mu_t$ is weakly continuous and
$$
\int^T_0\!\!\int_{\mR^d}\!\int_{\mR^d}|b_t(x,y)|\mu_t(\dif y)\mu_t(\dif x)\dif t<\infty,\  \ \forall T>0.
$$
\item $\mu_t(\dif y)=\rho^X_t(y)\dif y$ and $(t,y)\mapsto\rho^X_t(y)$ is continuous on $(0,\infty)\times\mR^d$ and 
satisfies the following two-sided estimate: for any $T>0$, there are constants $\gamma_0, c_0\geq 1$
such that for all $t\in(0,T]$ and $y\in\mR^d$,
$$
c_0^{-1}P_{t/\gamma_0}\mu_0(y)\leq \rho^X_t(y)\leq c_0P_{\gamma_0 t}\mu_0(y),
$$
where $P_t\mu_0(y):=(2\pi t)^{-d/2}\int_{\mR^d}\e^{-|x-y|^2/(2t)}\mu_0(\dif x)$ is the Gaussian heat semigroup.

\item If $\div b=0$, then for each $t>0$, $\rho^X_t(\cdot)\in C^1(\mR^d)$ and we have the following gradient estimate:
for any $T>0$, there are constants $\gamma_1, c_1\geq 1$ such that for all $t\in(0,T]$ and $y\in\mR^d$,
$$
|\nabla\rho^X_t(y)|\leq c_1 t^{-1/2}P_{\gamma_1 t}\mu_0(y).
$$
\end{enumerate}
\et
\bx
Let $b_t(x,y):=a_t(x,y)/|x-y|^\alpha$ for some $\alpha\in[1,2)$, where $a_t(x,y):\mR_+\times\mR^d\times\mR^d\to\mR^d$ satisfies that for some $\kappa>0$,
$$
|a_t(x,y)|\leq \kappa |x-y|. 
$$
Then it is easy to see that $b$  satisfies {\bf (H$^b$)} for some $p>d$ and $q=\infty$.
\ex

\br
Here an open question is to show the following propagation of chaos (see \cite{Sz}): Given $N\in\mN$, let $X^{N,j}, j=1,\cdots, N$ solve the following SDEs
$$
\dif X^{N,j}_t=\frac{1}{N}\sum_{i=1}^Nb_t(X^{N,j}_t,X^{N,i}_t)\dif t+\sqrt{2}\dif W^j_t,\ \ j=1,\cdots, N,
$$
where $W^j_\cdot, j=1,\cdots,N$ are $N$-independent $d$-dimensional Brownian motion. Let $X$ be the unique solution of SDE \eqref{SDE1} in Theorem \ref{Th11}.
Is it possible to show that 
$$
X^{N,1}_\cdot\to X_\cdot\mbox{ in distribution as $N\to\infty$?}
$$
Even for bounded measurable $b$, the above question seems to be still open.
\er

To show the existence of a solution to DDSDE \eqref{SDE1}, by the well-known result for bounded measurable drift $b$ obtained in \cite{Mi-Ve} 
(see also \cite{Li-Mi}and \cite{Zh1}), for each $n\in\mN$, there is a solution to the following distribution dependent SDE:
\begin{align}\label{SDE11}
\dif X^n_t=\left(\int_{\mR^d}b^n_t(X^n_t,y)\mu_{X^{n}_t}(\dif y)\right)\dif t+\sqrt{2}\dif W_t,\ \ X^n_0=X_0,
\end{align}
where $b^n_t(x,y):=(-n)\vee b_t(x,y)\wedge n.$
By the well-known results in \cite{Xi-Xi-Zh-Zh}, one can show the following uniform Krylov estimate: For any $p_1,q_1\in(1,\infty)$ with $\frac{d}{p_1}+\frac{2}{q_1}<2$ and $T>0$,
there is a constant $C>0$ such that for any $f\in\widetilde\mL^{p_1}_{q_1}(T)$,
\begin{align}\label{Kr0}
\sup_n\bE\left(\int^T_0 f_t(X^n_t)\dif t\right)\leq C_T\nor f\nor_{\widetilde\mL^{p_1}_{q_1}(T)}.
\end{align}
By this estimate and Zvonkin's technique, we can further show the tightness of $X^n_\cdot$ in the space of continuous functions. 
However, since $b$ is allowed to be singular,
it is not obvious by taking the limit $n\to\infty$ to obtain the existence of a solution. 
Indeed, one needs the following Krylov estimate: for suitable $p_0,q_0\in(1,\infty)$ and any $f:\mR_+\times\mR^d\times\mR^d\to\mR_+$,
$$
\sup_n\bE\left(\int^t_0f_s(X^n_s,\tilde X^{n}_s)\dif s\right)\leq \nor f\nor_{\widetilde\mL^{p_0}_{q_0}(T)},
$$
where $\tilde X^{n}_\cdot$ is an independent copy of $X^n$. When $b$ is bounded measurable, such an estimate is easy to get 
by considering $(X^n, \tilde X^{n})$ as an $\mR^{2d}$-dimensional
It\^o process and using the classical Krylov estimates (see \cite{Mi-Ve}).
While for singular $b$, such simple observation fails in order to obtain best integrability index $p$. 
We overcome this difficulty by a simple duality argument (see Lemma \ref{Le27} below). 
Moreover, concerning  the uniqueness, under assumption \eqref{TTV}, we shall employ Girsanov's transformation as usual.

\medskip

This paper is organized as follows: In Section 2, we prepare some well-known results and tools for later use.
In Section 3, we show the existence of weak and strong solutions to DDSDE \eqref{SDE} when the drift satisfies {\bf (H$^b$)},
and the diffusion coefficient is uniformly nondegenerate and bounded H\"older continuous.
In Section 4, we prove the uniqueness of weak and strong solutions to \eqref{SDE} in two cases: the coefficients $b$ and $\sigma$ are Lipschitz continuous
in the third variable with respect to the Wasserstein metric; drift $b$ is Lipschitz continuous in the third variable with respect to the total variation distance
and the diffusion coefficient does not depend on the distribution.
In Section 5, we present some applications to nonlinear FPE  \eqref{Non} and prove Theorem \ref{Th11}.

\medskip

Finally we collect some frequently used notations and conventions for later use.
\begin{itemize}
\item For $\theta>0$, $\cP_\theta(\mR^d):=\left\{\mu\in\cP(\mR^d): \int_{\mR^d}|x|^\theta\mu(\dif x)<\infty\right\}$.
\item For $R>0$, set $B_R:=\{x\in\mR^d: |x|<R\}$.
\item For a function $f:\mR^d\to\mR$, $\cM_R f(x):=\sup_{r\in(0,R)}\frac{1}{|B_r|}\int_{B_r}|f|(x+y)\dif y$.
\item Let $\bS_{\rm toch}$ be the set of all measurable stochastic processes on $(\Omega,\sF,\bP)$ that are stochastically continuous.
\item Let $b:\mR_+\times\mR^d\times\cP(\mR^d)\to\mR^d$ be a measurable vector field. For $X\in\bS_{\rm toch}$, define
\begin{align}\label{No1}
b^X_t(x):=b_t(x,\mu_{X_t}),\ \ \mu_{X_t}:=\bP\circ X_t^{-1}.
\end{align}
\item For a signed measure $\mu$, we denote by $\|\mu\|_{TV}:=\sup_{\|f\|_\infty\leq 1}|\mu(f)|$  the total variation of $\mu$.
\item For $j=1,2$, we introduce the index set $\sI_j$ as following:
\begin{align}\label{In}
\sI_j:=\Big\{(p,q)\in(1,\infty): \tfrac{d}{p}+\tfrac{2}{p}<j\Big\}.
\end{align}
\item For a matrix $\sigma$, we use $\|\sigma\|_{HS}$ to denote the Hilbert-Schmidt norm of $\sigma$.
\item We use $A\lesssim B$ (resp. $\asymp$) to denote $A\leq CB$ (resp. $C^{-1}B\leq A\leq CB$) for some unimportant constant $C\geq 1$, 
whose dependence on the parameters can be traced from the context. 
\end{itemize}

\section{Preliminaries}

In this section we recall some well-known results. We first introduce the following spaces and notations for later use.
For $(\alpha,p)\in\mR_+\times(1,\infty)$,  the usual Bessel potential space $H^{\alpha,p}$ is defined by
$$
H^{\alpha,p}:=\big\{f\in L^1_{loc}(\mR^d): \|f\|_{\alpha,p}:=\|(\mI-\Delta)^{\alpha/2}f\|_p<\infty\big\},
$$
where $\|\cdot\|_p$ is the usual $L^p$-norm, and $(\mI-\Delta)^{\alpha/2}f$ is defined by Fourier  transform
$$
(\mI-\Delta)^{\alpha/2}f:=\cF^{-1}\big((1+|\cdot|^2)^{\alpha/2}\cF f\big).
$$
Notice that for $n\in\mN$, an equivalent norm in $H^{n,p}$ is given by
$$
\|f\|_{n,p}=\|f\|_p+\|\nabla^n f\|_{p}.
$$
For $T>S\geq 0$, $p,q\in(1,\infty)$ and $\alpha\in\mR_+$, we introduce space-time function spaces
$$
\mL^p_q(S,T):=L^q\big([S,T];L^p\big),\ \  \mH^{\alpha,p}_q(S,T):=L^q\big([S,T];H^{\alpha,p}\big).
$$
Let $\chi\in C^\infty_c(\mR^d)$ be a smooth function with $\chi(x)=1$ for $|x|\leq 1$ and $\chi(x)=0$ for $|x|>2$.
For $r>0$ and $z\in\mR^d$, define
\begin{align}\label{CHI}
\chi^z_r(x):=\chi((x-z)/r).
\end{align}
Fix $r>0$. We introduce the following localized $H^{\alpha,p}$-space:
\begin{align}\label{Ck}
\widetilde H^{\alpha,p}:=\Big\{f\in H^{\alpha,p}_{loc}(\mR^d),\nor f\nor_{\alpha,p}:=\sup_z\|f\chi^z_r\|_{\alpha,p}<\infty\Big\},
\end{align}
and  the localized space-time function space $\widetilde\mH^{\alpha,p}_q(S,T)$ with norm
\begin{align}\label{GG1}
\nor f\nor_{\widetilde\mH^{\alpha,p}_q(S,T)}:=\sup_{z\in\mR^d}\|\chi^z_r f\|_{\mH^{\alpha,p}_q(S,T)}<\infty.
\end{align}
For simplicity we shall write
$$
\widetilde\mH^{\alpha,p}_q(T):=\widetilde\mH^{\alpha,p}_q(0,T),\ \ \widetilde\mL^{p}_q(T):=\widetilde\mH^{0,p}_q(0,T),
$$
and
$$
\widetilde\mH^{\alpha,p}_q:=\cap_{T>0}\widetilde\mH^{\alpha,p}_q(T),\ \ \widetilde\mL^{p}_q:=\cap_{T>0}\widetilde\mL^{p}_q(T).
$$

The following lemma list some easy properties of $\widetilde\mH^{\alpha,p}_q$ (see \cite{Zh-Zh2} and \cite{Xi-Xi-Zh-Zh}).
\bp\label{Pr41}
Let $p,q\in(1,\infty)$, $\alpha\in\mR_+$ and $T>0$.
\begin{enumerate}[{\rm(i)}]
\item For $r\not=r'>0$, there is a $C=C(d,\alpha,r,r',p,q)\geq 1$ such that
\begin{align}\label{GT1}
C^{-1}\sup_{z}\|f\chi^{z}_{r'}\|_{\mH^{\alpha,p}_q(T)}\leq \sup_{z}\|f\chi^{z}_r\|_{\mH^{\alpha,p}_q(T)}\leq C \sup_{z}\|f\chi^{z}_{r'}\|_{\mH^{\alpha,p}_q(T)}.
\end{align}
In other words, the definition of $\widetilde \mH^{\alpha,p}_q$ does not depend on the choice of $r$.
\item  Let $\alpha>0$ , $p,q\in[1,\infty)$ and $p'\in[p,\tfrac{pd}{d-p\alpha}\1_{p\alpha<d}+\infty\cdot\1_{p\alpha>d}]$. 
It holds that for some $C=C(d,\alpha,p,p')>0$,
\begin{align}\label{Sob}
\nor f\nor_{\widetilde\mL^{p'}_q(T)}\leq C\nor f\nor_{\widetilde\mH^{\alpha,p}_q(T)}.
\end{align}
\item For any $k\in\mN$, there is a constant $C=C(d,k,\alpha,p,q)\geq 1$ such that
$$
C^{-1} \nor f\nor_{\widetilde\mH^{\alpha+k,p}_q(T)}\leq \nor f\nor_{\widetilde\mH^{\alpha,p}_q(T)}+\nor \nabla^k f\nor_{\widetilde\mH^{\alpha,p}_q(T)}\leq 
C \nor f\nor_{\widetilde\mH^{\alpha+k,p}_q(T)}.
$$
\item Let $(\rho_\eps)_{\eps\in(0,1)}$ be a family of mollifiers in $\mR^d$ and $f_\eps(t,x):=f(t,\cdot)*\rho_\eps(x)$. 
For any $f\in\widetilde \mH^{\alpha,p}_q$, it holds that $f_\eps\in L^q_{loc}(\mR; C^\infty_b(\mR^d))$ and
for some $C=C(d,\alpha,p,q)>0$,
\begin{align}\label{GT2}
\nor f_\eps\nor_{\widetilde\mH^{\alpha,p}_q(T)}\leq C\nor f\nor_{\widetilde\mH^{\alpha,p}_q(T)},\ \forall \eps\in(0,1),
\end{align}
and for any $\varphi\in C^\infty_c(\mR^{d})$,
\begin{align}\label{GT3}
\lim_{\eps\to0}\|(f_\eps-f)\varphi\|_{\mH^{\alpha,p}_q(T)}=0.
\end{align}
\item For $r=p/(p-1)$ and $s=q/(q-1)$,
\begin{align}
&\nor f\nor_{\widetilde\mL^p_q(T)}\asymp \nor f\nor'_{\widetilde\mL^p_q(T)}
=\sup_{ \nor g\nor^*_{\widetilde\mL^{r}_{s}(T)}\leq 1}\left|\int^T_0\!\!\!\int_{\mR^d} f_t(x)g_t(x)\dif x\dif t\right|,\label{KH01}\\
&\qquad\mbox{\rm and }\nor g\nor^*_{\widetilde\mL^{r}_{s}(T)}
=\sup_{\nor f\nor'_{\widetilde\mL^p_q(T)} \leq 1}\left|\int^T_0\!\!\!\int_{\mR^d} f_t(x)g_t(x)\dif x\dif t\right|,\label{KH02}
\end{align}
where $\nor f\nor'_{\widetilde\mL^p_q(T)}:=\sup_{z\in\mZ^d}\|\1_{Q_z} f\|_{\mL^p_q(T)}$ and $\nor g\nor^*_{\widetilde\mL^{r}_{s}(T)}:=\sum_{z\in\mZ^d}\|\1_{Q_z} g\|_{\mL^{r}_{s}(T)}$,
\begin{align}\label{QQZ}
Q_z:=\Pi_{i=1}^d(z_i, z_i+1],\quad z=(z_1,\cdots,z_d)\in\mZ^d.
\end{align}
\end{enumerate}
\ep
\begin{proof}
The first four conclusions can be found in \cite[Proposition 4.1]{Zh-Zh2}.
We only prove (v). The equivalence between $\nor f\nor_{\widetilde\mL^p_q(T)}$ and $ \nor f\nor'_{\widetilde\mL^p_q(T)}$ is obvious by definition. 
Concerning the others,  we note that by H\"older's inequality,
\begin{align}\label{KH3}
\begin{split}
&\int^T_0\!\!\!\int_{\mR^d} f_t(x)g_t(x)\dif x\dif t=\sum_{z\in\mZ^d}\int^T_0\!\!\!\int_{\mR^d}\1_{Q_z}(x) f_t(x)g_t(x)\dif x\dif t\\
&\qquad\leq \sum_{z\in\mZ^d}\|\1_{Q_z}f\|_{\mL^p_q(T)}\|\1_{Q_z} g\|_{\mL^{r}_{s}(T)}\leq\nor f\nor'_{\widetilde\mL^p_q(T)}\nor g\nor^*_{\widetilde\mL^{r}_{s}(T)}.
\end{split}
\end{align}
On the other hand, assume that $z_n$ is a sequence in $\mZ^d$ so that for $Q_n:=Q_{z_n}$,
\begin{align}
\lim_{n\to\infty}\|\1_{Q_{n}} f\|_{\mL^p_q(T)}=\nor f\nor'_{\widetilde\mL^p_q(T)}.\label{KH33}
\end{align}
If we take
$$
g_t(x):=\frac{\1_{Q_n}(x)  |f_t(x)|^{p-1}}{\|\1_{Q_n}f_t\|^{p-q}_{p}}\left(\int^T_0\|\1_{Q_n} f_t\|^{q}_{p}\dif t\right)^{1/{q}-1}
$$
with the convention $0/0=0$, then by easy calculations, we have
$\nor g\nor^*_{\widetilde\mL^{r}_{s}(T)}=1$ and
\begin{align*}
\int^T_0\!\!\!\int_{\mR^d}f_t(x)g_t(x)\dif x\dif t=\left(\int^T_0\|\1_{Q_n}f_t\|^{q}_{p}\dif t\right)^{1/{q}}=\|\1_{Q_n} f\|_{\mL^p_q(T)},
\end{align*}
which together with \eqref{KH3} and \eqref{KH33} yields \eqref{KH01}.
Similarly,  if we take
$$
f_t(x):=\sum_{z\in\mZ^d}\frac{\1_{Q_z}(x)  |g_t(x)|^{r-1}}{\|\1_{Q_z}g_t\|^{r-s}_{r}}\cdot \left(\int^T_0\|\1_{Q_z} g_t\|^{s}_{r}\dif t\right)^{1/{s}-1},
$$
then $\nor f\nor'_{\widetilde\mL^p_q(T)}=1$ and
\begin{align*}
\int^T_0\!\!\!\int_{\mR^d}f_t(x)g_t(x)\dif x\dif t=\sum_{z\in\mZ^d}\left(\int^T_0\|\1_{Q_z}g_t\|^{s}_{r}\dif t\right)^{1/{s}}=\nor g\nor^*_{\widetilde\mL^{r}_{s}(T)},
\end{align*}
which together with \eqref{KH3} yields \eqref{KH02}.
\end{proof}

We now recall the following result about $L^q(L^p)$-solvability of PDE (see \cite{Xi-Xi-Zh-Zh}).
\bt\label{pde}
Let $(p,q)\in\sI_1$ (see \eqref{In}) and $T>0$.
Assume that $\sigma_t(x,\mu)=\sigma_t(x)$ and $b_t(x,\mu)=b_t(x)$ are independent of $\mu$, 
and satisfy that for some $c_0\geq 1$, $\gamma\in(0,1]$ and for all $t\geq 0$, $x,y,\xi\in\mR^d$,
\begin{align}\label{SI}
c_0^{-1}|\xi|\leq |\sigma_t(x)\xi|\leq c_0|\xi|,\ \ 
\|\sigma_t(x)-\sigma_t(y)\|_{HS}\leq c_0|x-y|^\gamma,
\end{align}
and $\nor b\nor_{\widetilde\mL^{p}_{q}(T)}\leq\kappa_0$ for some $\kappa_0>0$,
Then for any $\lambda\geq 1$ and $f\in \widetilde\mL^p_q(T)$, there exists a unique solution $u\in \widetilde\mH^{2,p}_q(T)$ 
to the following backward parabolic equation:
\begin{align}
\p_tu+(\sL^\sigma_t-\lambda)u+b\cdot\nabla u=f,\quad u(T,x)=0.\label{pide55}
\end{align}
Moreover, letting $\Theta:=(\gamma,c_0,d,p,q,\kappa_0, T)$,
we have the following:
\begin{enumerate}[{\rm (i)}]
\item For any $\alpha\in[0,2-\frac{2}{q})$,
there is a $c_1=c_1(\alpha,\Theta)>0$ such that for all $\lambda\geq 1$,
\begin{align}\label{Max}
\lambda^{1-\frac{\alpha}{2}-\frac{1}{q}}\nor u\nor_{\widetilde\mH^{\alpha,p}_\infty(T)}
+\nor u\nor_{\widetilde\mH^{2,p}_q(T)}\leq C\nor f\nor_{\widetilde\mL^p_q(T)}.
\end{align}
\item Let $(\sigma',b',f')$ be another set of coefficients satisfying the same assumptions as $(\sigma,b,f)$ with the same parameters $(\gamma,c_0,\kappa_0)$. 
Let $u'$ be the solution of \eqref{pide55}
corresponding to $(\sigma',b',f')$.
For any $\alpha\in[0,2-\frac{2}{q})$,
there is a constant $c_2=c_2(\alpha,\Theta)>0$ such that  for all $\lambda\geq 1$,
\begin{align}\label{es10}
\begin{split}
&\lambda^{1-\frac{\alpha}{2}-\frac{1}{q}}\nor u-u'\nor_{\widetilde\mH^{\alpha,p}_\infty(T)}\leq c_2\nor f-f'\nor_{\widetilde\mL^p_q(T)}\\
&\quad+ c_2\|f\|_{\widetilde\mL^p_q(T)}\big(\|\sigma-\sigma'\|_{ \mL^\infty(T)}+\nor b-b'\nor_{\widetilde\mL^{p}_{q}(T)}\big).   
\end{split}
\end{align}
\end{enumerate}
\et
\begin{proof}
The existence and uniqueness of $u\in\widetilde\mH^{2,p}_q(T)$ as well as the first conclusion are proved in \cite[Theorem 3.1]{Xi-Xi-Zh-Zh}. We only show (ii).
Let $w=u'-u$. Then 
$$
\p_t w+(\sL^{\sigma'}_t-\lambda) w+b'\cdot\nabla w=(\sL^{\sigma}_t-\sL^{\sigma'}_t)u+(b-b')\cdot\nabla u+f'-f.
$$
By \eqref{Max} and H\"older's inequality we have
\begin{align*}
&\lambda^{1-\frac{\alpha}{2}-\frac{1}{q}}\nor w\nor_{\widetilde\mH^{\alpha,p}_\infty(T)}
\lesssim \nor(\sL^{\sigma}_t-\sL^{\sigma'}_t)u+(b-b')\cdot\nabla u+f'-f\nor_{\widetilde\mL^p_q(T)}\\
&\qquad\lesssim \|\sigma'-\sigma\|_{\mL^\infty(T)}\nor\nabla^2u\nor_{\widetilde\mL^p_q(T)}+\nor b'-b\nor_{\widetilde\mL^{p}_{q}(T)}
\cdot\|\nabla u\|_{\mL^\infty(T)}+\nor f'-f\nor_{\widetilde\mL^p_q(T)}.
\end{align*}
Estimate \eqref{es10} now follows by Sobolev's embedding \eqref{Sob} due to $\frac{d}{p}+\frac{2}{q}<1$ and \eqref{Max}. 
\end{proof}

\br
It should be noted that if $b$ is bounded measurable, then the assertions in Theorem \ref{pde} holds for all $p,q\in(1,\infty)$.
\er

The following stochastic Gronwall inequality for continuous martingales was proved by Scheutzow \cite{Sc}, and
for general discontinuous martingales in \cite{Xi-Zh}.

\bl[Stochastic Gronwall's inequality]\label{im}
Let $\xi(t)$ and $\eta(t)$ be two nonnegative c\`adl\`ag $\sF_t$-adapted processes, 
$A_t$ a continuous nondecreasing $\sF_t$-adapted process with $A_0=0$, $M_t$ a  local martingale with $M_0=0$. Suppose that
\begin{align}\label{Gron}
\xi(t)\leq\eta(t)+\int^t_0\xi(s)\dif A_s+M_t,\ \forall t\geq 0.
\end{align}
Then for any $0<q<p<1$ and $\tau>0$, we have
\begin{align}
\big[\bE(\xi(\tau)^*)^{q}\big]^{1/q}\leq \Big(\tfrac{p}{p-q}\Big)^{1/q}\Big(\bE \e^{pA_{\tau}/(1-p)}\Big)^{(1-p)/p}\bE\big(\eta(\tau)^*\big),  \label{gron}
\end{align}
where $\xi(t)^*:=\sup_{s\in[0,t]}\xi(s)$.
\el

We also recall the following result about maximal functions (see \cite[Lemma 2.1]{Xi-Xi-Zh-Zh}).
\bl\label{Le2}
\begin{enumerate}[{\rm (i)}]
\item For any $R>0$, there exists a constant $C=C(d,R)$ such that for any $f\in L^\infty(\mR^d)$ with $\nabla f\in L^1_{loc}(\mR^d)$
and Lebesgue-almost all $x,y\in \mR^d$,
\begin{align}
|f(x)-f(y)|\leq C |x-y|(\cM_R|\nabla f|(x)+\cM_R|\nabla f|(y)+\|f\|_\infty),\label{ES2}
\end{align}
where $\cM_R$ is defined at the end of the introduction.
\item For any $p>1$ and $R>0$, there is a constant $C=C(R,d,p)$ such that for any $T>0$ and all $f\in\widetilde\mL^p_q(T)$,
\begin{align}\label{GW1}
\nor\cM_R f\nor_{\widetilde\mL^p_q(T)}\leq C\nor f\nor_{\widetilde\mL^p_q(T)}.
\end{align}
\end{enumerate}
\el
We introduce the following notion about Krylov's estimates.
\bd
Let $p,q\in(1,\infty)$ and $T,\kappa>0$. We say a stochastic process $X\in\bS_{\rm toch}$ satisfies Krylov's estimate with index $p,q$ and constant $\kappa$ if for any 
$f\in \widetilde\mL^p_q(T)$,
\begin{align}\label{Kry}
\bE\left(\int^T_0f_t(X_t)\dif t\right)\leq \kappa\nor f\nor_{\widetilde\mL^p_q(T)}.
\end{align}
The set of all such $X$ will be denoted by ${\bf K}^{p,q}_{T,\kappa}$.
\ed
For a space-time function $f_t(x,y):\mR_+\times\mR^d\times\mR^d\to\mR$ and $p_1,p_2,q_0\in[1,\infty]$, we also introduce the norm
$$
\nor f\nor_{\widetilde\mL^{p_1,p_2}_{q_0}(T)}:=\sup_{z,z'\in\mR^d}\left(\int^T_0\left(\int_{Q_{z'}}\|\1_{Q_z}f_t(\cdot,y)\|^{p_2}_{p_1}\dif y\right)^{\frac{q_0}{p_2}}\right)^{\frac{1}{q_0}}.
$$

The following lemma is an easy consequence of Proposition \ref{Pr41} (v).
\bl\label{Le27}
Let $p_1,p_2, q_0,q_1,q_2\in(1,\infty)$ with $\frac{1}{q_1}+\frac{1}{q_2}=1+\frac{1}{q_0}$ and $T,\kappa_1,\kappa_2>0$. 
Let $X\in\bK^{p_1,q_1}_{T,\kappa_1}$ and $Y\in\bK^{p_2,q_2}_{T,\kappa_2}$ be two independent processes.
Then for any $f_t(x,y)\in\widetilde\mL^{p_1,p_2}_{q_0}(T)$,
\begin{align}\label{GQ2}
\bE\left(\int^T_0f_t(X_t, Y_t)\dif t\right)&\leq \kappa_1\kappa_2\nor f\nor_{\widetilde\mL^{p_1,p_2}_{q_0}(T)}.
\end{align}
\el
\begin{proof}
Let $Z^1=X$ and $Z^2=Y$. First of all, by Krylov's estimate \eqref{Kry}, for each $i=1,2$, there is a function $\rho^{Z^i}\in\mL^{r_i}_{s_i}(T)$ with $r_i=\frac{p_i}{p_i-1}$, $s_i=\frac{q_i}{q_i-1}$ so that
$$
\int^T_0\!\!\!\int_{\mR^d}f_t(x)\rho^{Z^i}_t(x)\dif x\dif t=\bE\left(\int^T_0f_t(Z^i_t)\dif t\right)\leq \kappa_i\nor f\nor_{\widetilde\mL^{p_i}_{q_i}(T)}\leq \kappa_i\|f\|_{\mL^{p_i}_{q_i}(T)}.
$$
By Proposition \ref{Pr41} (v), we further have
$$
\nor\rho^{Z^i}\nor^*_{\widetilde\mL^{r_i}_{s_i}(T)}:=\sum_{z\in\mZ^d}\|\1_{Q_z}\rho^{Z^i}\|_{\mL^{r_i}_{s_i}(T)}\leq \kappa_i,\ \ i=1,2,
$$
where $Q_z$ is defined by \eqref{QQZ}.
Now by the independence of $X,Y$ and H\"older's inequality, we have 
\begin{align*}
&\bE\left(\int^T_0f_t(X_t, Y_t)\dif t\right)=\int^T_0\!\!\!\int_{\mR^d}\!\int_{\mR^d} f_t(x, y)\rho^X_t(x)\rho^Y_t(y)\dif x\dif y\dif t\\
&\qquad=\sum_{z\in\mZ^d}\sum_{z'\in\mZ^d}\int^T_0\!\!\!\int_{\mR^d}\!\int_{\mR^d}\1_{Q_z}(x)\1_{Q_{z'}}(y) f_t(x, y)\rho^X_t(x)\rho^Y_t(y)\dif x\dif y\dif t\\
&\qquad\leq\sum_{z\in\mZ^d}\sum_{z'\in\mZ^d}\|\1_{Q_z\times Q_{z'}} f\|_{\mL^{p_1,p_2}_{q_0}(T)}\|\1_{Q_z}\rho^X\|_{\mL^{r_1}_{s_1}(T)}\|\1_{Q_{z'}}\rho^Y\|_{\mL^{r_2}_{s_2}(T)}\\
&\qquad\leq \kappa_1\kappa_2\sup_{z, z'\in\mZ^d}\|\1_{Q_z\times Q_{z'}} f\|_{\mL^{p_1,p_2}_{q_0}(T)}=\kappa_1\kappa_2\nor f\nor_{\widetilde\mL^{p_1,p_2}_{q_0}(T)},
\end{align*}
which gives \eqref{GQ2}. The proof is complete.
\end{proof}

Now we  prove the following convergence lemmas, which have independent interest and will be crucial for showing the existence of solutions in Section 3.
\bl\label{Le32}
Let $X^n,Y^n, X, Y\in\bS_{\rm toch}$ be such that for each $t\geq 0$,
$X^n_t$ converges to $X_t$ almost surely and $Y^n_t$ converges to $Y_t$ in distribution.
Let $p,q>1$ and $T,\beta,\kappa>0$. Suppose that $X^n\in\bK^{p,q}_{T,\kappa}$ for each $n\in\mN$,
and for some $C_1>0$,
\begin{align}\label{Kry1}
\sup_n\sup_{t\in[0,T]}\bE|X^n_t|^\beta\leq C_1.
\end{align}
If for each $(t,x)$, $\mu\mapsto b_t(x,\mu)$ is continuous with respect to the weak convergence topology and 
for some $\gamma>1$, $C_2>0$ and all $Z\in\bS_{\rm toch}$,
\begin{align}\label{JW1}
\nor b^Z\nor_{\mL^{\gamma p}_{\gamma q}(T)}\leq C_2,
\end{align}
where $b^Z$ is defined by \eqref{No1}, then
\begin{align}\label{ER4}
\lim_{n\to\infty}\bE\left(\int^T_0|b^{Y_n}_t(X^n_t)-b^Y_t(X_t)|\dif t\right)=0.
\end{align}
\el
\begin{proof}
To prove \eqref{ER4}, it suffices to show the following:
\begin{align}
\lim_{n\to\infty}\bE\left(\int^T_0\left|b^{Y^n}_t(X^n_t)-b^Y_t(X^n_t)\right|\dif t\right)=0,\label{ER1}\\
\lim_{n\to\infty}\bE\left(\int^T_0\left|b^Y_t(X^n_t)-b^Y_t(X_t)\right|\dif t\right)=0.\label{ER2}
\end{align}
We first look at \eqref{ER1}.
Since $\mu_{Y^n_t}$ weakly converges to $\mu_{Y_t}$ for each $t\geq 0$, by the assumption we have
\begin{align}\label{LK1}
b^{Y_n}_t(x)\stackrel{n\to\infty}{\to} b^Y_t(x),\ \ \forall (t,x)\in\mR_+\times\mR^d.
\end{align}
For fixed $R,M>0$, since $X^n\in\bK^{p,q}_{T,\kappa}$ (see \eqref{Kry}), by the definitions we have
\begin{align*}
&\bE\left(\int^T_0\1_{B_R}(X^n_t)\left|b^{Y_n}_t(X^n_t)-b^Y_t(X^n_t)\right|\dif t\right)
\leq \kappa\nor\1_{B_R}(b^{Y_n}-b^Y)\nor_{\widetilde\mL^p_q(T)}\\
&\quad\lesssim\left\|\1_{B_R}(b^{Y_n}-b^Y)\1_{|b^{Y_n}-b^Y|\leq M}\right\|_{\mL^p_q(T)}
+\left\|\1_{B_R}(b^{Y_n}-b^Y)\1_{|b^{Y_n}-b^Y|>M}\right\|_{\mL^p_q(T)}\\
&\quad\leq \left\|\1_{B_R}(b^{Y_n}-b^Y)\1_{|b^{Y_n}-b^Y|\leq M}\right\|_{\mL^p_q(T)}
+\left\|\1_{B_R}|b^{Y_n}-b^Y|^\gamma\right\|_{\mL^p_q(T)}/M^{\gamma-1}.
\end{align*}
By the dominated convergence theorem and \eqref{LK1}, the first term converges to zero as $n\to\infty$ for each $M>0$.
By \eqref{JW1}, the second term converges to zero uniformly in $n$ as $M\to\infty$. Thus, we obtain that for any $R>0$,
\begin{align}\label{Lim1}
\lim_{n\to\infty}\bE\left(\int^T_0\1_{B_R}(X^n_t)\left|b^{Y_n}_t(X^n_t)-b^Y_t(X^n_t)\right|\dif t\right)=0.
\end{align}
On the other hand, by H\"older and Chebyshev's inequalities and \eqref{Kry1}, we have
\begin{align*}
&\bE\left(\int^T_0\1_{B^c_R}(X^n_t)\left|b^{Y_n}_t(X^n_t)-b^Y_t(X^n_t)\right|\dif t\right)\\
&\leq \int^T_0\bP(|X^n_t|>R)^{\frac{\gamma-1}{\gamma}}\left(\bE\left|b^{Y_n}_t(X^n_t)-b^Y_t(X^n_t)\right|^\gamma\right)^{\frac{1}{\gamma}}\dif t\\
&\leq\sup_{t\in[0,T]}\bP(|X^n_t|>R)^{\frac{\gamma-1}{\gamma}} T^{\frac{\gamma-1}{\gamma}}
\left(\int^T_0\bE\left|b^{Y_n}_t(X^n_t)-b^Y_t(X^n_t)\right|^\gamma\dif t\right)^{\frac{1}{\gamma}}\\
&\leq\left(\frac{C_1 T}{R^\beta}\right)^{\frac{\gamma-1}{\gamma}} 
\kappa^{\frac{1}{\gamma}}\nor b^{Y_n}-b^Y\nor_{\mL^{\gamma p}_{\gamma q}(T)}
\stackrel{\eqref{JW1}}{\leq} \left(\frac{C_1 T}{R^\beta}\right)^{\frac{\gamma-1}{\gamma}} \kappa^{\frac{1}{\gamma}}\cdot 2C_2.
\end{align*}
Combining this with \eqref{Lim1}, we obtain \eqref{ER1}.

Next we show \eqref{ER2}. Let $b^{Y,\eps}_t(x):=b^Y_t(\cdot)*\varrho_\eps(x)$ be a mollifying approximation of $b^Y$. 
By Proposition \ref{Pr41} (iv) and \eqref{Kry1}, as above one can derive that
\begin{align}\label{ET3}
\lim_{\eps\to 0}\sup_{n\in\mN\cup\{\infty\}}\bE\left(\int^T_0\left|b^{Y,\eps}_t(X^n_t)-b^Y_t(X^n_t)\right|\dif t\right)=0,
\end{align}
where we have used the convention $X^\infty:=X$.
On the other hand, since by \eqref{Kry},
$$
\sup_n\bE\left(\int^T_0\left|b^{Y,\eps}_t(X^n_t)-b^{Y,\eps}_t(X_t)\right|^\gamma\dif t\right)\leq C\nor b^{Y,\eps}\nor^\gamma_{\widetilde\mL^{\gamma p}_{\gamma q}(T)},
$$
and for fixed $\eps>0$ and any $t>0$, $x\mapsto b^{Y,\eps}_t(x)$ is continuous, by the dominated convergence theorem, we have
$$
\lim_{n\to\infty}\bE\left(\int^T_0\left|b^{Y,\eps}_t(X^n_t)-b^{Y,\eps}_t(X_t)\right|\dif t\right)=0,
$$
which together with \eqref{ET3} yields \eqref{ER2}. 
\end{proof}

There are, of course, many examples where the {\it weak} continuity assumption of $\mu\mapsto b_t(x,\mu)$ in the above lemma
is not satisfied, as in the following interesting case:
\begin{align}\label{LK2}
b_t(x,\mu)=\int_{\mR^d}{\bar b}_t(x,y)\mu(\dif y),
\end{align}
where ${\bar b}:\mR_+\times\mR^d\times\mR^d\to\mR$ is a bounded measurable function.
Obviously the weak continuity of $\mu\mapsto b(t,x,\mu)$ does not hold.
However, in this case we still have the following limiting result.
\bl\label{Le33}
Let $X^n, Y^n, X, Y\in\bS_{\rm toch}$ be such that for each $t\geq 0$, $X^n_t$ converges to $X_t$ almost surely and $Y^n_t$ converges to $Y_t$ in distribution.
Let $p_1,p_2, q_0,q_1,q_2\in(1,\infty)$ with $\frac{1}{q_1}+\frac{1}{q_2}=1+\frac{1}{q_0}$ and $T,\beta,\kappa>0$. 
Suppose that $X^n\in\bK^{p_1,q_1}_{T,\kappa}$ and $Y^n\in\bK^{p_2,q_2}_{T,\kappa}$ for each $n\in\mN$, and that
there is a constant $C_1>0$ such that
\begin{align}\label{LK3}
\sup_n\sup_{t\in[0,T]}\bE\left(|X^n_t|^\beta+|Y^n_t|^\beta\right)\leq C_1.
\end{align}
Let $\gamma>1$. Then for any $\bar b\in \widetilde\mL^{\gamma p_1,\gamma p_2}_{\gamma q_0}(T)$, we have 
\begin{align}\label{ER44}
\lim_{n\to\infty}\bE\left(\int^T_0|b^{Y^n}_t(X^n_t)-b^{Y}_t(X_t)|\dif t\right)=0.
\end{align}\el
\begin{proof}
Let $\mN_\infty:=\mN\cup\{\infty\}$ and $Y^\infty:=Y, X^\infty:=X$. Since $b^{Y^n}$ only depends on the distribution of $Y^n$,
by Skorohod's representation,
without loss of generality we may assume that $(X^n)_{n\in\mN_\infty}$ and $(Y^n)_{n\in\mN_\infty}$ are independent, 
and $(X^n_t,Y^n_t)\to (X_t,Y_t)$ a.e. as $n\to\infty$ for each $t$.
Notice that by the assumptions and \eqref{GQ2},
\begin{align}\label{KG1}
\sup_{n\in\mN_\infty}\bE\left(\int^T_0|\bar b_t(X^n_t, Y^n_t)|^\gamma\dif t\right)
\leq \kappa^2\nor \bar b\nor^\gamma_{\widetilde\mL^{\gamma p_1,\gamma p_2}_{\gamma q_0}(T)}<\infty.
\end{align}
Let $\bar b^\eps_t(x,y)=\bar b_t*\varrho_\eps(x,y)$ be a mollifying approximation of $\bar b$. As in the proof of \eqref{ER1}, we have
\begin{align}\label{QA1}
\lim_{\eps\to 0}\sup_{n\in\mN_\infty}\bE\left(\int^T_0|\bar b^\eps_t(X^n_t, Y^n_t)-\bar b_t(X^n_t, Y^n_t)|\dif t\right)=0.
\end{align}
Thus, to prove \eqref{ER44}, it suffices to show that for fixed $\eps\in(0,1)$,
\begin{align*}
&\lim_{n\to\infty}\bE\left(\int^T_0|{\bar b}^\eps_t(X^n_t,Y^n_t)-{\bar b}^\eps_t(X^n_t,Y_t)|\dif t\right)=0,\\
&\lim_{n\to\infty}\bE\left(\int^T_0|{\bar b}^\eps_t(X^n_t,Y_t)-{\bar b}^\eps_t(X_t,Y_t)|\dif t\right)=0,
\end{align*}
which follows by \eqref{KG1} and the dominated convergence theorem.
\end{proof}

\section{Existence of weak and strong solutions}

In this section we show the weak existence and strong existence of DDSDEs with singular drifts. 
First of all we recall the notions of martingale solutions and weak solutions for \eqref{SDE}.
Let $\mC$ be the space of all continuous functions from $\mR_+$ to $\mR^d$, which is endowed with the usual Borel $\sigma$-field $\cB(\mC)$.
The set of all probability measures on $(\mC,\cB(\mC))$ is denoted by $\cP(\mC)$. 
Let $w_t$ be the coordinate process over $\mC$, that is,
$$
w_t(\omega)=\omega_t,\ \ \omega\in\mC.
$$
For $t\geq 0$, let $\cB_t(\mC)=\sigma\{w_s: s\leq t\}$ be the natural filtration.
For a probability measure $\mP\in\cP(\mC)$, the expectation with respect to $\mP$ will be denoted by $\mE$ if there is no confusion. 
\bd[Martingale solutions]\label{Def2}
We call a probability measure $\mP\in\cP(\mC)$ a martingale solution of DDSDE \eqref{SDE} with initial distribution $\nu\in\cP(\mR^d)$
if $\mP\circ w_0^{-1}=\nu$ and for any $f\in C^\infty(\mR^d)$,
$$
\int^{t}_0|\sL^{\sigma^\mP}_s f|(w_s)\dif s+\int^t_0|b^\mP_s\cdot\nabla f|(w_s)\dif s<\infty,\ \ \mP-a.s,\ \ \forall t>0,
$$
where $\sigma^\mP_t(x):=\sigma_t(x,\mu^\mP_t)$ and $b^\mP_t(x):=b_t(x,\mu^\mP_t)$, $\mu^\mP_t:=\mP\circ w_t^{-1}$, and
\begin{align}\label{EQ1}
M^f_t:=f(w_{t})-f(w_0)-\int^{t}_0(\sL^{\sigma^\mP}_s f)(w_s)\dif s-\int^t_0 (b^\mP_s\cdot\nabla f)(w_s)\dif s,
\end{align}
is a continuous local $\cB_t(\mC)$-martingale under $\mP$.
All the martingale solutions of DDSDE \eqref{SDE} with coefficients $\sigma,b$ and initial distribution $\nu$ 
are denoted by $\sM^{\sigma,b}_\nu$.
\ed

\bd[Weak solutions]
\label{Def99}
 Let $(X,W)$ be two $\mR^d$-valued continuous adapted processes on some 
 filtered probability space $(\Omega,\sF, (\sF_t)_{t\geq 0}, \bP)$. 
 We call 
 $$
 (\Omega,\sF, (\sF_t)_{t\geq 0}, \bP; X,W)
 $$ 
 a weak solution of DDSDE \eqref{SDE} with initial distribution $\nu\in\cP(\mR^d)$ if 
\begin{enumerate}[{\rm (i)}]
\item $\bP\circ X^{-1}_0=\nu$ and $W$ is a $d$-dimensional standard $\sF_t$-Brownian motion.
\item For all $t>0$, it holds that
$$
\int^t_0|b_s|(X_s,\mu_{X_s})\dif s+\int^t_0\|\sigma_s\sigma^*_s\|_{HS}(X_s,\mu_{X_s})\dif s<\infty,\ \ \bP-a.s.
$$
and
\begin{align}\label{Def9}
X_t=X_0+\int^t_0b_s(X_s,\mu_{X_s})\dif s+\int^t_0\sigma_s(X_s,\mu_{X_s})\dif W_s,\ \ \bP-a.s.
\end{align}
\end{enumerate}
\ed
\br\rm
It is well known that weak solutions and martingale solutions are equivalent (cf. \cite{St-Va}), which means that for any $\mP\in\sM^{\sigma,b}_\nu$, there is a 
weak solution 
$$
(\Omega,\sF, (\sF_t)_{t\geq 0}, \bP; X,W)
$$ 
to DDSDE \eqref{SDE} with initial distribution $\nu\in\cP(\mR^d)$ such that
$$
\mP=\bP\circ X^{-1}.
$$
\er

Now we make the following assumptions about $\sigma$ and $b$:
\begin{enumerate}[{\bf (H$^{\sigma,b}$)}]
\item For each $t,x$, the mapping $\mu\mapsto \sigma_t(x,\mu)$ is weakly continuous,
and there are $c_0\geq 1$ and $\gamma\in(0,1]$ such that for all $t\geq 0$, $x,x',\xi\in\mR^d$ and $\mu\in\cP(\mR^d)$,
\begin{align}\label{SIG}
\qquad\quad c_0^{-1}|\xi|\leq |\sigma_t(x,\mu)\xi|\leq c_0|\xi|,\ 
\|\sigma_t(x,\mu)-\sigma_t(x',\mu)\|_{HS}\leq c_0|x-x'|^\gamma.
\end{align}
The drift $b$ satisfies one of the following conditions:
\begin{enumerate}[{\rm (i)}]
\item For each $t,x$, the mapping $\mu\mapsto b_t(x,\mu)$ is weakly continuous, and for some $(p,q)\in\sI_1$ and $\kappa_0>0$, 
\begin{align}\label{GR2}
\sup_{Z\in\bS_{\rm toch}}\nor b^Z\nor_{\widetilde\mL^{p}_{q}(T)}\leq\kappa_0<\infty.
\end{align}
\item $b$ has the form \eqref{LK2} with $\bar b$ satisfying {\bf (H$^b$)}.
\end{enumerate}
\end{enumerate}

It should be noticed that under {\bf (H$^{b}$)}, \eqref{GR2} holds. Indeed, by definition we have
\begin{align*}
\nor b^Z\nor^q_{\widetilde\mL^{p}_{q}(T)}
&=\sup_{z\in\mR^d}\int^T_0\left\|\chi^z_r\int_{\mR^d}\bar b_s(\cdot,y)\mu_{Z_s}(\dif y)\right\|_p^q\dif s\\
&\leq\sup_{z\in\mR^d}\int^T_0\left\|\chi^z_r\int_{\mR^d}h_s(\cdot-y)\mu_{Z_s}(\dif y)\right\|_p^q\dif s\\
&\leq\sup_{z\in\mR^d}\int^T_0\!\!\int_{\mR^d}\left\|\chi^{z-y}_r h_s\right\|_p^q\mu_{Z_s}(\dif y)\dif s\\
&\leq\int^T_0\sup_{z\in\mR^d}\left\|\chi^{z}_r h_s\right\|_p^q\mu_{Z_s}(\mR^d)\dif s=\int^T_0\nor h_s\nor_p^q\dif s.
\end{align*}

To show the existence of weak solutions, we first establish the following apriori estimates.
\bl\label{Le35}
Let $\beta>0$. Under {\bf (H$^{\sigma,b}$)}, for any $\nu\in\cP_\beta(\mR^d)$ and $Z\in\bS_{\rm toch}$, there is a unique weak solution 
$(\Omega,\sF, (\sF_t)_{t\geq 0}, \bP; X,W)$ 
to the following SDE:
$$
\dif X_t=b^Z_t(X_t)\dif t+\sigma^Z_t(X_t)\dif W_t,\ \ \bP\circ X^{-1}_0=\nu.
$$
Moreover, letting $\Theta=(d,p,q,c_0,\gamma, \kappa_0,\beta)$, we have
\begin{enumerate}[{\rm (i)}]
\item For any $T>0$, there is a $C_1=C_1(\Theta,T)>0$ such that 
\begin{align}\label{PK10}
\bE\left(\sup_{t\in[0,T]}|X_t|^\beta\right)\leq C_1(\bE|X_0|^{\beta}+1),
\end{align}
and for any $\delta<T$,
\begin{align}\label{PK1}
\bE\left(\sup_{t\in[0,T-\delta]}|X_{t+\delta}-X_t|^\beta\right)\leq C_1\delta^{\beta/2}.
\end{align}
\item  For any $(p_1,q_1)\in\sI_2$ and
$T>0$,  there is a constant $C_2=C_2(p_1,q_1,\Theta,T)>0$ such that  for all $0\leq t_0<t_1\leq T$ and $f\in\widetilde\mL^{p_1}_{q_1}(t_0,t_1)$,
\begin{align}\label{Kry2}
\bE\left(\int^{t_1}_{t_0} f_s(X_s)\dif s\Big|\sF_{t_0}\right)\leq C_2\nor f\nor_{\widetilde\mL^{p_1}_{q_1}(t_0,t_1)}.
\end{align}
\end{enumerate}
\el
\begin{proof}
The proof of this lemma is essentially contained in \cite{Zh-Zh}. For the reader's convenience, we sketch the proofs below.
We use Zvonkin's transformation to kill the drift $b^Z$.  For $\lambda, T>0$, consider the following backward PDE:
$$
\p_t u+(\sL^{\sigma^Z}_t-\lambda) u+b^Z\cdot\nabla u+b^Z=0,\ u(T,x)=0.
$$
Since $b^Z\in\widetilde\mL^{p}_{q}(T)$ with $(p,q)\in\sI_1$, by Theorem \ref{pde}, for $\lambda\geq 1$,
there is a unique solution $u\in\widetilde\mH^{2,p}_{q}(T)$ solving the above PDE. Moreover, for any $\alpha\in[0,2-\frac{2}{q})$,
there is a constant $c_1=c_1(\alpha,\Theta,T)>0$ such that for all $\lambda\geq 1$,
\begin{align}\label{HF11}
\lambda^{1-\frac{\alpha}{2}-\frac{2}{q}}\nor u\nor_{\widetilde\mH^{\alpha,p}_\infty(T)}
+\nor u\nor_{\widetilde\mH^{2,p}_{q}(T)}\leq c_1\nor b^Z\nor_{\widetilde\mL^{p}_{q}(T)}. 
\end{align}
In particular, since $\frac{d}{p}+\frac{2}{q}<1$, by \eqref{Sob} we can choose $\lambda$ large enough so that
$$
\|u\|_{\mL^{\infty}(T)}+\|\nabla u\|_{\mL^{\infty}(T)}\leq 1/2. 
$$
Now if we define 
$$
\Phi_t(x):=x+u_t(x),
$$
then it is easy to see that
\begin{align}\label{LK6}
|x-y|/2\leq |\Phi_t(x)-\Phi_t(y)|\leq 2|x-y|,
\end{align}
and
\begin{align}\label{ER5}
\p_t\Phi+\sL^{\sigma^Z}_t\Phi+b^Z\cdot\nabla\Phi=\lambda u.
\end{align}
By the generalized It\^o formula and \eqref{ER5}, we have
\begin{align}
Y_t:=\Phi_t(X_t)&=\Phi_0(X_0)+\lambda\int^t_0u_s(X_s)\dif s+\int^t_0(\sigma^Z_s\cdot\nabla\Phi_s)(X_s)\dif W_s,\no\\
&=\Phi_0(X_0)+\int^t_0\tilde b_s(Y_s)\dif s+\int^t_0\tilde\sigma_s(Y_s)\dif W_s,\label{LK7}
\end{align}
where 
$$
\tilde\sigma:=(\sigma^Z\cdot\nabla\Phi)\circ\Phi^{-1},\  \ \tilde b:=\lambda u\circ\Phi^{-1}.
$$ 
Moreover, by \eqref{HF11}, \eqref{LK6} and the Sobolev embedding \eqref{Sob}, it is easy to see that for some 
$c_2=c_2(\Theta, T)>0$ and $\gamma_0=\gamma_0(\gamma,p,q)\in(0,1)$,
\begin{align}\label{HF2}
c_2^{-1}|\xi|\leq |\tilde\sigma_t(x)\xi|\leq c_2|\xi|,\ \ 
\|\tilde\sigma_t(x)-\tilde\sigma_t(y)\|_{HS}\leq c_2|x-y|^{\gamma_0},
\end{align}
and
\begin{align}\label{HF20}
\|\tilde b\|_{\mL^\infty(T)}+\|\nabla\tilde b\|_{\mL^\infty(T)}\leq 4\lambda.
\end{align}
By well-known results, SDE \eqref{LK7} admits a unique weak solution (cf. \cite{St-Va}).
Moreover, as in \cite{Zh-Zh}, one can check that $X_t:=\Phi^{-1}_t(Y_t)$ solves the original SDE.

(i) Let $\beta>0$. By \eqref{HF2} and \eqref{HF20}, estimate \eqref{PK10} directly follows by BDG's inequality. We prove \eqref{PK1}.
Fix $\delta\in(0,T)$. Let $\tau$ be any stopping time less than $T-\delta$. By equation \eqref{LK7} and  BDG's inequality, we have
\begin{align*}
\bE|Y_{\tau+\delta}-Y_\tau|^\beta&\lesssim \bE\left|\int^{\tau+\delta}_{\tau}\tilde b_s(X_s)\dif s\right|^\beta
+\bE\left|\int^{\tau+\delta}_{\tau}\tilde\sigma_s(X_s)\dif W_s\right|^\beta\\
&\lesssim\|\tilde b\|^\beta_{\mL^\infty(T)}\delta^\beta+\|\tilde \sigma\|^\beta_{\mL^\infty(T)}\delta^{\beta/2}\leq C\delta^{\beta/2},
\end{align*}
which yields \eqref{PK1}  by \cite[Lemma 2.7]{Zh-Zh1} and \eqref{LK6}.

(ii) It was proved in \cite[Theorem 2.1]{Zh0}
(see also \cite[Theorem 5.7]{Xi-Zh}) that for any $(p_1,q_1)\in\sI_2$, there is a constant $C_2=C_2(p_1,q_1,\Theta,T)>0$ such that  
for all $0\leq t_0<t_1\leq T$ and $f\in\widetilde\mL^{p_1}_{q_1}(t_0,t_1)$,
$$
\bE\left(\int^{t_1}_{t_0} f_s(Y_s)\dif s\Big|\sF_{t_0}\right)\leq C_2\nor f\nor_{\widetilde\mL^{p_1}_{q_1}(t_0,t_1)}.
$$
By a change of variable and \eqref{LK6} again, we obtain \eqref{Kry2}.
\end{proof}

\br
An important conclusion of (ii) above is the following 
Khasminskii's type estimate (see \cite[Lemma 3.5]{Xi-Zh}): For any $\lambda, T>0$ and $f\in\widetilde\mL^{p_1}_{q_1}(T)$ with $(p_1,q_1)\in\sI_2$,
\begin{align}\label{Kh}
\bE\exp\left(\lambda\int^T_0|f_s(X_s)|\dif s\right)\leq C_3,
\end{align}
where $C_3$ only depends on $\lambda,\Theta,p_1,q_1,T$ and $\nor f\nor_{\widetilde\mL^{p_1}_{q_1}(T)}$.
\er

Now we can  show the following weak existence result.
\begin{theorem}
Let $\beta>2$. Under {\bf (H$^{\sigma,b}$)}, for any $\nu\in\cP_\beta(\mR^d)$, 
there exists a weak solution $(\Omega,\sF, (\sF_t)_{t\geq 0}, \bP; X,W)$ to DDSDE \eqref{SDE} with $\bP\circ X^{-1}_0=\nu$. 
\end{theorem}
\begin{proof}
Let $X^0_t\equiv X_0$. For $n\in\mN$, consider the following approximating SDE:
\begin{align}
X^n_t=X^n_0 +\int^t_0b^n_s(X^n_s, \mu_{X^n_s})\dif s+\int^t_0\sigma_s(X^n_s, \mu_{X^n_s})\dif W_s,\label{SDE0}
\end{align}
where 
$$
b^n_s(x,\mu):=(-n)\vee b_s(x,\mu)\wedge n,\ \ \bar b^n_s(x,y):=(-n)\vee \bar b_s(x,y)\wedge n.
$$
Since $b^n$ is bounded measurable, by \cite{Mi-Ve} or \cite[Theorem 1.2]{Zh1}, there is a weak solution 
$$
(\Omega,\sF, (\sF_t)_{t\geq 0}, \bP; X^n,W)
$$ 
to DDSDE \eqref{SDE0}  with $\bP\circ (X^n_0)^{-1}=\nu$. Moreover, 
since
$$
\sup_{Z\in\bS_{\rm toch}}\nor b^{n,Z}\nor_{\widetilde\mL^{p}_{q}(T)}\leq\sup_{Z\in\bS_{\rm toch}}\nor b^{Z}\nor_{\widetilde\mL^{p}_{q}(T)}\leq\kappa_0,
$$
by Lemma \ref{Le35}, the following uniform estimates hold:
\begin{enumerate}[{\rm (i)}]
\item For any $T>0$, there is a constant $C_1>0$ such that 
$$
\sup_n\bE\left(\sup_{t\in[0,T]}|X^n_t|^\beta\right)\leq C_1(\bE|X_0|^{\beta}+1),
$$
and for all $\delta\in(0,T)$,
$$
\sup_n\bE\left(\sup_{t\in[0,T-\delta]}|X^n_{t+\delta}-X^n_t|^\beta\right)\leq C_1\delta^{\beta/2}.
$$
\item Let $(p_1,q_1)\in\sI_2$. For any $T>0,$ there is a $C_2>0$ such that 
for all $f\in\widetilde\mL^{p_1}_{q_1}(T)$,
$$
\sup_n\bE\left(\int^T_0 f_s(X^n_s)\dif s\right)\leq C_2\nor f\nor_{\widetilde\mL^{p_1}_{q_1}(T)}.
$$
\end{enumerate}
Now by (i), the laws $\mQ^n$ of $(X^n, W)$ in $\mC\times\mC$ are tight. Let $\mQ$ be any accumulation point of $\mQ^n$.
Without loss of generality, we assume that $\mQ^n$ weakly converges to some probability measure $\mQ$.
By Skorokhod's representation theorem,  there are a probability space $(\tilde\Omega,\tilde\sF,\tilde\bP)$ and 
random variables $(\tilde X^n , \tilde W^n)$ and $(\tilde X,\tilde W)$ defined on it such that
\begin{align}\label{FD4}
(\tilde X^n , \tilde W^n)\to (\tilde X,\tilde W),\ \ \tilde\bP-a.s.
\end{align}
and
\begin{align}\label{FD5}
\tilde\bP\circ(\tilde X^n , \tilde W^n)^{-1}=\mQ^n=\bP\circ(X^n, W)^{-1},\quad
\tilde\bP\circ(\tilde X, \tilde W)^{-1}=\mQ.
\end{align}
Define $\tilde\sF^n _t:=\sigma(\tilde W^n _s, \tilde X^n_s;s\leq t)$. 
We note that
\begin{align*}
&\bP(W _t-W _s\in\cdot |\sF _s)=\bP(W _t-W _s\in\cdot)\\
&\Rightarrow
\tilde\bP(\tilde W^n _t-\tilde W^n _s\in\cdot |\tilde \sF^n _s)=\tilde\bP(\tilde W^n _t-\tilde W^n _s\in\cdot).
\end{align*}
In other words, $\tilde W^n $ is an $\tilde\sF_t^n $-Brownian motion. Thus, by \eqref{SDE0} and \eqref{FD5} we have
$$
\tilde X^n _t=\tilde X^n_0
+\int^t_0b^n_s(\tilde X^n_s,\mu_{\tilde X^n_s})\dif s+\int^t_0\sigma_s(\tilde X^n _s,\mu_{\tilde X^n_s})\dif \tilde W^n _s.
$$
By (ii), \eqref{FD4}, Lemmas \ref{Le32}, \ref{Le33} and \cite[Theorem 6.22, p383]{Ja-Sh}, one can take limits as $n\to\infty$ to obtain
$$
\tilde X_t=\tilde X_0+\int^t_0b_s(\tilde X_s,\mu_{\tilde X_s})\dif s+\int^t_0\sigma_s(\tilde X_s,\mu_{\tilde X_s})\dif \tilde W_s.
$$
Here we only check that the assumptions of Lemma \ref{Le33} are satisfied in the case that
$b$ takes the form \eqref{LK2} with $\bar b$ satisfying {\bf (H$^b$)}. Clearly, by (ii) above, for any $(p_1,q_1)\in\sI_2$,
there is a $\kappa>0$ such that for each $n\in\mN$,
$$
\tilde X^n\in\bK^{p_1,q_1}_{T,\kappa}.
$$
We note that $|\bar b_t(x,y)|\leq h_t(x-y)$, where for some $(p,q)\in\sI_1$,
$h\in L^q_{loc}(\mR_+;\widetilde L^p(\mR^d))\subset\widetilde\mL^p_q.$ One can choose $\gamma>1$ so that
$\tfrac{d\gamma }{p}+\tfrac{\gamma}{q}<1.$ Now if we take $p_1=p_2=\frac{p}{\gamma}$, $q_0=\frac{q}{\gamma}$, 
$q_1=q_2=\frac{2q}{q+\gamma}$, then it is easy to see that
$(p_1,q_1)\in\sI_2$ and 
$$
\bar b\in\widetilde\mL^{p,\infty}_{q}=\widetilde\mL^{\gamma p_1,\infty}_{\gamma q_0}\subset\cap_{p'\geq 1}\widetilde\mL^{\gamma p_1,p'}_{\gamma q_0}.
$$
Thus one can apply Lemma \ref{Le33} to conclude that
$$
\lim_{n\to\infty}\tilde\bE\left(\int^t_0|b_s(\tilde X^n_s,\mu_{\tilde X^n_s})-b_s(\tilde X_s,\mu_{\tilde X_s})|\dif s\right)=0.
$$
Moreover, as in showing \eqref{QA1}, we also have
$$
\lim_{m\to\infty}\sup_n\tilde\bE\left(\int^t_0|\bar b^m_s-\bar b_s|(\tilde X^n_s,Y^n_s)\dif s\right)=0,
$$
where $Y^n_\cdot$ is an independent copy of $\tilde X^n_\cdot$.
The proof is thus complete.
\end{proof}

About the existence of strong solutions, we have

\bc\label{Co38}
Let $\beta>2$. Under {\bf (H$^{\sigma,b}$)}, if for some $(p_1,q_1)\in\sI_1$, 
$$\sup_{Z\in\bS_{\rm toch}}\nor\nabla\sigma^Z\nor_{\widetilde\mL^{q_1}_{p_1}(T)}<\infty,$$
then for any initial random variable $X_0$ with finite $\beta$-order moment, 
there exists a strong solution to DDSDE \eqref{SDE}.
\ec
\begin{proof}
Let $(\Omega,\sF, (\sF_t)_{t\geq 0}, \bP; X,W)$ be a weak solution of DDSDE \eqref{SDE}. Define
$$
b^X_t(x):=b_t(x,\mu_{X_t}),\ \ \sigma^X_t(x):=\sigma_t(x,\mu_{X_t}),\ \mu_{X_t}:=\bP\circ X^{-1}_t.
$$
Consider the following SDE:
$$
\dif Z_t=b^X_t(Z_t)\dif t+\sigma^X_t(Z_t)\dif W_t.
$$
Under the assumption of the theorem, it has been shown in \cite{Xi-Xi-Zh-Zh} that there is a unique strong solution to this equation.
Since $X$ also satisfies the above equation, by strong uniqueness, we obtain that $X=Z$ is a strong solution.
\end{proof}

\br
Although we have shown the existence of strong or weak solutions, the uniqueness of strong solutions or weak solutions is a more difficult problem.
\er

\section{Uniqueness of strong and weak solutions}

In this section we study the uniqueness of strong and weak solutions. We introduce the following assumptions about the dependence on third variable $\mu$:
\begin{enumerate}[\bf (A$^{\sigma,b}_\theta$)]
\item Let $(p,q), (p_1,q_1)\in\sI_1$ and $\theta\geq 1$. It holds that 
$$
\sup_{Z\in\bS_{\rm toch}}\nor b^Z\nor_{\widetilde\mL^{p}_q(T)}<\infty,\ \ \sup_{Z\in\bS_{\rm toch}}\nor\nabla\sigma^Z\nor_{\widetilde\mL^{p_1}_{q_1}(T)}<\infty,
$$
and there are $\ell\in L^q_{loc}(\mR_+)$ and a constant $c_0\geq 1$ such that
for any two random variables $X,Y$ with finite $\theta$-order moments,
\begin{align}\label{DG2}
\begin{split}
&\nor b_t(\cdot,\mu_X)-b_t(\cdot,\mu_Y)\nor_p\leq \ell_t\|X-Y\|_{\theta},\ \\
&\|\sigma_t(\cdot,\mu_X)-\sigma_t(\cdot,\mu_Y)\|_\infty\leq c_0\|X-Y\|_{\theta},
\end{split}
\end{align}
where $\|\cdot\|_\theta$ stands for the $L^\theta$-norm in the probability space $(\Omega,\sF,\bP)$.
\end{enumerate}
Notice that \eqref{DG2} is equivalent to that for all $\mu,\mu'\in\cP_\theta(\mR^d)$,
\begin{align*}
\begin{split}
&\|b_t(\cdot,\mu)-b_t(\cdot,\mu')\|_p\leq \ell_t\cW_\theta(\mu,\mu'),\\ 
&\|\sigma_t(\cdot,\mu)-\sigma_t(\cdot,\mu')\|_\infty\leq c_0\cW_\theta(\mu,\mu'),
\end{split}
\end{align*}
where $\cW_\theta$ is the usual Wasserstein metric of $\theta$-order.
For convenience, we would like to use \eqref{DG2} rather than introducing the Wasserstein metric.
\br
We note that in \cite{Hu-Wa}, \eqref{DG2} is assumed to hold for $p=\infty$.
\er
We first show the following strong uniqueness result.
\bt\label{Th51}
Let $\theta\geq 1$ and $\beta>2\vee\theta$. Under {\bf (H$^{\sigma,b}$)} and {\bf (A$^{\sigma,b}_\theta$)},
for any initial random variable $X_0$ with finite $\beta$-order moment,
there is a unique strong solution to DDSDE \eqref{SDE}.
\et
\begin{proof}
Below we fix  $p,q\in\sI_1$, and 
without loss of generality, we consider the time interval $[0,1]$ and assume that for some $\gamma>1$,
\begin{align}\label{GG}
\|\ell\|_{L^{\gamma q}(0,1)}+\sup_{Z\in\bS_{\rm toch}}\|b^Z\|_{\mL^{\gamma q}_{\gamma p}(1)}<\infty.
\end{align}
Otherwise, we may choose $\gamma>1$ so that $\frac{2\gamma}{q}+\frac{d\gamma}{p}<1$ holds and replace $(p,q)$ with $(p/\gamma,q/\gamma)$.
The existence of strong solutions has been shown in Corollary \ref{Co38}. We only need to prove the pathwise uniqueness.
Let $X,Y$ be two strong solutions defined on the same probability space with same starting points $X_0=Y_0$ a.s.
We divide the proof into three steps and use the convention that all the constants below will be independent of $T\in[0,1]$.

(i)  Let $T\in(0,1)$ and $\lambda>0$. We consider the following backward PDE:
\begin{align}\label{ET4}
\p_t u^X+(\sL^{\sigma^X}_t-\lambda) u+b^X\cdot\nabla u^X+b^X=0,\ u^X(T,x)=0.
\end{align}
By Theorem \ref{pde}, for $\lambda\geq 1$,
there is a unique solution $u^X\in\widetilde\mH^{2,p}_{q}(T)$ solving the above PDE. Moreover, for any $\alpha\in[0,2-\frac{2}{q})$,
there is a constant $c_1>0$ such that for all $\lambda\geq 1$ and $T\in[0,1]$,
\begin{align}\label{HF1}
\lambda^{1-\frac{\alpha}{2}-\frac{2}{q}}\nor u^X\nor_{\widetilde\mH^{\alpha,p}_\infty(T)}
+\nor u^X\nor_{\widetilde\mH^{2,p}_{q}(T)}\leq c_1\nor b^X\nor_{\widetilde\mL^{p}_{q}(T)}. 
\end{align}
In particular, since $\frac{d}{p}+\frac{2}{q}<1$, by \eqref{Sob}, we can choose $\lambda$ large enough so that
\begin{align}\label{ER8}
\|u^X\|_{\mL^{\infty}(T)}+\|\nabla u^X\|_{\mL^{\infty}(T)}\leq 1/2,\ \ \forall T\in[0,1]. 
\end{align}
Below we shall fix such a $\lambda$ and define 
$$
\Phi^X_t(x):=x+u^X_t(x).
$$ 
It is easy to see that
$$
\p_t\Phi^X+\sL^{\sigma^X}_t\Phi^X+b^X\cdot\nabla\Phi^X=\lambda u^X.
$$

(ii) By the generalized It\^o formula, we have
\begin{align}\label{KJ1}
\tilde X_t:=\Phi^X_t(X_t)=\Phi^X_0(X_0)+\lambda\int^t_0u^X_s(X_s)\dif s+\int^t_0\tilde\sigma^X_s(X_s)\dif W_s,
\end{align}
where 
$$
\tilde\sigma^X:=\sigma^X\cdot\nabla\Phi^X.
$$
Similarly, we define $\tilde Y_t:=\Phi^Y_t(Y_t)$,  and for simplicity write
$$
\xi_t:=X_t-Y_t,\ \ \tilde \xi_t:=\tilde X_t-\tilde Y_t.
$$
Noting that by \eqref{ER8},
$$
|x-y|\leq 2|\Phi^X_t(x)-\Phi^X_t(y)|\leq 2|\Phi^X_t(x)-\Phi^Y_t(y)|+2\|u^X-u^Y\|_{\mL^\infty(T)}
$$
and
$$
|\Phi^X_t(x)-\Phi^Y_t(y)|\leq 2|x-y|+\|u^X-u^Y\|_{\mL^\infty(T)},
$$
we have
\begin{align}\label{KJ2}
|\xi_t|\leq 2|\tilde\xi_t|+2\|u^X-u^Y\|_{\mL^\infty(T)},\ \ |\tilde\xi_t|\leq 2|\xi_t|+\|u^X-u^Y\|_{\mL^\infty(T)}.
\end{align}
By \eqref{KJ1} and again It\^o's formula, we have for any $\beta\geq 1$,
\begin{align*}
|\tilde\xi_t|^{\beta}&=|\tilde\xi_0|^{\beta}+\beta\lambda\int^t_0|\tilde\xi_s|^{\beta-2}\<\tilde\xi_s,u^X_s(X_s)-u^Y_s(Y_s)\>\dif s\\
&\quad+\beta\int^t_0|\tilde\xi_s|^{\beta-2}\<(\tilde\sigma^X_s(X_s)-\tilde\sigma^Y_s(Y_s))^*\tilde\xi_s, \dif W_s\>\\
&\quad+\beta\Big(\tfrac{\beta}{2}-1\Big)\int^t_0|\tilde\xi_s|^{\beta-4}|(\tilde\sigma^X_s(X_s)-\tilde\sigma^Y_s(Y_s))^*\tilde\xi_s|^2\dif s\\
&\quad+\frac{\beta}{2}\int^t_0|\tilde\xi_s|^{\beta-2}\|\tilde\sigma^X_s(X_s)-\tilde\sigma^Y_s(Y_s)\|^2_{HS}\dif s\\
&:=I_1+I_2+I_3+I_4+I_5.
\end{align*}
Since by \eqref{ER8},
\begin{align*}
|u^X_t(x)-u^Y_t(y)|\leq |x-y|+\|u^X-u^Y\|_{\mL^\infty(T)},
\end{align*}
by Young's inequality we obtain
\begin{align*}
I_2&\lesssim \int^t_0|\tilde\xi_s|^{\beta}\dif s+\lambda\int^t_0|u^X_s(X_s)-u^Y_s(Y_s)|^\beta\dif s\\
&\lesssim \int^t_0(|\tilde\xi_s|^{\beta}+\lambda |\xi_s|^{\beta})\dif s
+\lambda^{\beta}T\|u^X-u^Y\|_{\mL^\infty(T)}^{\beta}.
\end{align*}
Let 
$$
g^X_s(x):=|\nabla^2 u^X_s(x)|+|\nabla\sigma^X_s(x)|+\|\nabla u^X\|_{\mL^\infty(T)}+\|\sigma^X\|_{\mL^\infty(T)}.
$$ 
By the definition of $\tilde\sigma^X$, we also have that
\begin{align*}
&|\tilde\sigma^X_s(x)-\tilde\sigma^Y_s(y)|\\
&\leq \|\sigma^Y\|_{\mL^\infty(T)}|\nabla\Phi^X_s(x)-\nabla\Phi^Y_s(y)|
+|\sigma^X_s(x)-\sigma^Y_s(y)|\cdot\|\nabla\Phi^X\|_{\mL^\infty(T)}\\
&\leq \|\sigma^Y\|_{\mL^\infty(T)}\Big(|\nabla u^X_s(x)-\nabla u^X_s(y)|+|\nabla u^X_s(y)-\nabla u^Y(s,y)|\Big)\\
&\quad+\Big(|\sigma^X_s(x)-\sigma^X_s(y)|+|\sigma^X_s(y)-\sigma^Y_s(y)|\Big)\cdot\|\nabla\Phi^X\|_{\mL^\infty(T)}\\
&\stackrel{\eqref{ES2}}{\lesssim} |x-y|\Big(\cM_1 g^X_s(x)+\cM_1 g^X_s(y)\Big)+\|\nabla u^X-\nabla u^Y\|_{\mL^\infty(T)}+
\|\sigma^X_s-\sigma^Y_s\|_\infty.
\end{align*}
Hence, 
\begin{align*}
I_4+I_5&\lesssim \int^t_0\Big(|\xi_s|^\beta+|\tilde\xi_s|^\beta\Big)\Big(\cM g^X_s(X_s)+\cM g^X_s(Y_s)\Big)^2\dif s\\
&\quad+T\|\nabla u^X-\nabla u^Y\|^{\beta}_{\mL^\infty(T)}+\int^t_0\|\sigma^X_s-\sigma^Y_s\|_\infty^{\beta}\dif s.
\end{align*}
Combining the above calculations and noting that $|\tilde\xi_0|\leq \|u^X_0-u^Y_0\|_\infty$, we obtain
\begin{align}\label{KH1}
\begin{split}
|\tilde\xi_t|^\beta&\lesssim\|u^X-u^Y\|^\beta_{\mH^{1,\infty}_\infty(T)}
+\int^t_0\Big(|\tilde\xi_s|^\beta+|\xi_s|^\beta+\|\xi_s\|_\theta^\beta\Big)\dif s\\
&\quad+\int^t_0\Big(|\xi_s|^\beta+|\tilde\xi_s|^\beta\Big)\Big(\cM_1 g^X_s(X_s)+\cM_1 g^X_s(Y_s)\Big)^2\dif s+M_t,
\end{split}
\end{align}
where $M_t$ is a continuous local martingale.

(iii) Now we  define
$$
A_t:=t+\int^t_0\Big(\cM_1 g^X_s(X_s)+\cM_1 g^X_s(Y_s)\Big)^2\dif s.
$$
By \eqref{KH1} and \eqref{KJ2}, we obtain that for all $t\in[0,T]$,
\begin{align*}
|\xi_s|^\beta+|\tilde\xi_s|^\beta\lesssim \|u^X-u^Y\|^\beta_{\mH^{1,\infty}_\infty(T)}+\int^t_0\|\xi_s\|_\theta^\beta\dif s
+\int^t_0\Big(|\xi_s|^\beta+|\tilde\xi_s|^\beta\Big)\dif A_s+M_t.
\end{align*}
Note that by the assumption and \eqref{GW1},
$$
(s,x)\mapsto (\cM_1 |\nabla^2 u^X_s(x)|)^2\in \widetilde\mL^{p/2}_{q/2}(T),
$$ 
and
$$
(s,x)\mapsto (\cM_1 |\nabla\sigma^X_s(x)|)^2\in \widetilde\mL^{p_1/2}_{q_1/2}(T).
$$ 
Since $(\frac{p}{2},\frac{q}{2}), (\frac{p_1}{2},\frac{q_1}{2})\in\sI_2$, by Khasminskii's estimate \eqref{Kh},  we have
$$
\mE\exp{\gamma A_T}<\infty,\ \ \forall \gamma>0,\ \ \forall T\in[0,1].
$$
Thus we can use the stochastic Gronwall inequality \eqref{gron} to derive that
\begin{align}\label{ET6}
\sup_{s\in[0,T]}\|\xi_s\|_\theta^\beta=\left(\sup_{s\in[0,T]}\bE |\xi_s|^{\theta}\right)^{\beta/\theta}&\lesssim 
\|u^X-u^Y\|^\beta_{\mH^{1,\infty}_\infty(T)}+\int^T_0\|\xi_s\|_\theta^\beta\dif s.
\end{align}
Noticing that by \eqref{DG2},
$$
\nor b^X-b^Y\nor_{\widetilde\mL^p_q(T)}\leq \left(\int^T_0\ell^q_t\|X_t-Y_t\|^q_\theta\dif t\right)^{1/q}
\leq\|\ell\|_{L^q(0,T)}\sup_{t\in[0,T]}\|\xi_t\|_\theta,
$$
and
$$
\|\sigma^X-\sigma^Y\|_{\mL^\infty(T)}\leq c_0\sup_{t\in[0,T]}\|X_t-Y_t\|_\theta=c_0\sup_{t\in[0,T]}\|\xi_t\|_\theta,
$$
we have by \eqref{es10},
\begin{align*}
&\|u^X-u^Y\|_{\mH^{1,\infty}_\infty(T)}\\
&\lesssim \nor b^X-b^Y\nor_{\widetilde\mL^p_q(T)}+\nor b^X\nor_{\widetilde\mL^p_q(T)}\left(\|\sigma^X-\sigma^Y\|_{\mL^\infty(T)}
+\nor b^X-b^Y\nor_{\widetilde\mL^p_q(T)}\right)\\
&\lesssim \left(\|\ell\|_{L^q(0,T)}+\nor b^X\nor_{\widetilde\mL^p_q(T)}\right)\sup_{t\in[0,T]}\|\xi_t\|_\theta
\stackrel{\eqref{GG}}{\lesssim}T^{\frac{\gamma-1}{\gamma q}}\sup_{t\in[0,T]}\|\xi_t\|_\theta.
\end{align*}
Substituting this into \eqref{ET6}, we obtain
$$
\sup_{s\in[0,T]}\|\xi_s\|_\theta^\beta\leq CT^{\frac{\beta(\gamma-1)}{\gamma q}}\sup_{t\in[0,T]}\|\xi_t\|^\beta_\theta,\ \ T\in(0,1),
$$
where $C$ does not depend on $T\in(0,1)$.
By choosing $T$ small enough, we get $\|\xi_t\|^\beta_\theta=0$ for all $t\in[0,T]$.
By shifting the time $T$, we obtain the uniqueness.
\end{proof}

It is obvious that $b$ defined in \eqref{LK2} does not satisfy \eqref{DG2}. Below we shall relax it to the weighted total variation norm by Girsanov's transformation.
The price we have to pay is that we need to assume that the diffusion coefficient does not depend on the time marginal law of $X$. 
For $\theta\geq 1$, let 
$$
\phi_\theta(x):=1+|x|^\theta.
$$
We assume
\begin{enumerate}[{\bf ($\widetilde{\bf A}^{\sigma,b}_\theta$)}]
\item Let $(p,q), (p_1,q_1)\in\sI_1$ and $\theta\geq 1$ and $\sigma_t(x,\mu)=\sigma_t(x)$. It holds that 
$$
\sup_{Z\in\bS_{\rm toch}}\nor b^Z\nor_{\widetilde\mL^{p}_q(T)}<\infty,\ \ \nor\nabla\sigma\nor_{\widetilde\mL^{p_1}_{q_1}(T)}<\infty,
$$
and there is an $\ell\in L^q_{loc}(\mR_+)$ such that for all $\mu,\mu'\in\cP(\mR^d)$ and $t\geq 0$,
\begin{align}\label{CC1}
\begin{split}
\nor b(t,\cdot,\mu)-b(t,\cdot,\mu')\nor_p\leq \ell_t\|\phi_\theta\cdot (\mu-\mu')\|_{TV}.
\end{split}
\end{align}
\end{enumerate}
It should be noted that \cite[Theorem 6.15]{Vi} implies,
$$
\cW_\theta(\mu,\mu')\leq c\|\phi_\theta\cdot (\mu-\mu')\|^{1/\theta}_{TV}.
$$
\bt\label{Th43}
Let $\theta\geq 1$ and $\beta>2\theta$. Under {\bf (H$^{\sigma,b}$)} and {\bf ($\widetilde{\bf A}^{\sigma,b}_\theta$)},
for any initial random variable $X_0$ with finite $\beta$-order moment,
there is a unique weak solution to DDSDE \eqref{SDE}, which is also a unique strong solution.
\et
\begin{proof}
We use the Girsanov transform in the same way asin \cite{Mi-Ve} to show the weak uniqueness, and so also the strong uniqueness. 
Since under the assumptions of the theorem, weak solutions are also strong solutions (see Corollary \ref{Co38}),
without loss of generality, let $X^{(i)}, i=1,2$ be two solutions of SDE \eqref{SDE} defined on the same probability space $(\Omega,\sF,\bP)$
and with the same Brownian motion and starting point $\xi$. That is,
\begin{align}\label{EQ0}
\dif X^{(i)}_t=\sigma_t(X^{(i)}_t)\dif W_t+b_t(X^{(i)}_t,\mu^{(i)}_t)\dif t,\ \ X^{(i)}_0=\xi,
\end{align}
where $\mu^{(i)}_t=\bP\circ (X^{(i)}_t)^{-1}$. We want to show $\mu^{(1)}_t=\mu^{(2)}_t$.

Since $\sigma_t(x,\mu)=\sigma_t(x)$ satisfies \eqref{SI} under our assumptions, it is well known that there is a unique weak solution to SDE
$$
\dif Z_t=\sigma_t(Z_t)\dif W_t,\ \ Z_0=\xi.
$$
Let $\beta>2\theta$. Since $\sigma$ is bounded, it is easy to see that
\begin{align}\label{Mo1}
\sup_{t\in[0,T]}\bE|Z_t|^\beta\leq C\Big(\bE|\xi|^\beta+1\Big).
\end{align}
Define
$$
\tilde b^{(i)}_s(x):=\sigma^{-1}_s(x)\cdot b^{X^{(i)}}_s(x),\ \ \tilde W^{(i)}_t:=W_t-\int^t_0\tilde b^{(i)}_s(Z_s)\dif s,
$$
and
$$
\sE^{(i)}_T:=\exp\left\{\int^T_0\tilde b^{(i)}_s(Z_s)\cdot\dif W_s-\frac{1}{2}\int^T_0|\tilde b^{(i)}_s(Z_s)|^2\dif s\right\}.
$$
Since $\nor\tilde b^{(i)}\nor_{\widetilde\mL^p_q(T)}\leq\nor b^{X^{(i)}}\nor_{\widetilde\mL^p_q(T)}<\infty$ for some $(p,q)\in\sI_1$, 
by Khasminskii's estimate \eqref{Kh}, we have
\begin{align}\label{ET5}
\bE\exp\left\{\gamma\int^T_0|\tilde b^{(i)}_s(Z_s)|^2\dif s\right\}\leq C_{T,\gamma},\ \ \forall \gamma>0,
\end{align}
and for any $\gamma\in\mR$,
\begin{align}\label{ET50}
\bE (\sE^{(i)}_T)^\gamma\leq C_{T,\gamma}<\infty.
\end{align}
Hence, for each $i=1,2$, $\bE\sE^{(i)}_T=1$, and $\tilde W^{(i)}$ is still a Brownian motion under $\sE^{(i)}_T\cdot\bP$, and
$$
\dif Z_t=\sigma_t(Z_t)\dif\tilde W^{(i)}_t+b^{X^{(i)}}_t(Z_t)\dif t,\ \ Z_0=\xi.
$$
Since the above SDE admits a unique strong solution (see also \eqref{EQ0}), we have
$$
(\sE^{(i)}_T\bP)\circ Z_{T}^{-1}=\bP\circ (X^{(i)}_{T})^{-1}=\mu^{(i)}_T, \ \ i=1,2.
$$
Therefore, for $\delta=\frac{\beta}{\beta-\theta}<2$, by H\"older's inequality, we get
\begin{align}
&\|\phi_\theta\cdot(\mu^{(1)}_T-\mu^{(2)}_T)\|_{TV}=\|\phi_\theta\cdot ((\sE^{(1)}_T\bP)\circ Z_{T}^{-1}-(\sE^{(2)}_T\bP)\circ Z_{T}^{-1})\|_{TV}\no\\
&\qquad\leq\bE\Big(\phi_\theta(Z_T)|\sE^{(1)}_T-\sE^{(2)}_T|\Big)\leq\|\phi_\theta(Z_T)\|_{\delta/(\delta-1)}\|\sE^{(1)}_T-\sE^{(2)}_T\|_{\delta}\no\\
&\qquad=\|1+|Z_T|^\theta\|_{\beta/\theta}\|\sE^{(1)}_T-\sE^{(2)}_T\|_{\delta}
\stackrel{\eqref{Mo1}}{\leq} C\|\sE^{(1)}_T-\sE^{(2)}_T\|_{\delta}.\label{LK9}
\end{align}
Noting that
$$
\dif \sE^{(i)}_t=\sE^{(i)}_t\tilde b^{(i)}_t(Z_t)\cdot\dif W_t,
$$
we have
$$
\dif (\sE^{(1)}_t-\sE^{(2)}_t)=(\sE^{(1)}_t\tilde b^{(1)}_t(Z_t)-\sE^{(2)}_t\tilde b^{(2)}_t(Z_t))\cdot\dif W_t.
$$
By It\^o's formula, we have
\begin{align*}
&\dif |\sE^{(1)}_t-\sE^{(2)}_t|^2=|\sE^{(1)}_t\tilde b^{(1)}_t(Z_t)-\sE^{(2)}_t\tilde b^{(2)}_t(Z_t)|^2\dif t+M_t,\\
&\quad\leq 2|\sE^{(1)}_t-\sE^{(2)}_t|^{2}|\tilde b^{(1)}_t(Z_t)|^2\dif t+2|\sE^{(2)}_t(\tilde b^{(1)}_t(Z_t)-\tilde b^{(2)}_t(Z_t))|^2\dif t+M_t,
\end{align*}
where $M$ is a continuous local martingale.
Since $\delta<2$, by the stochastic Gronwall inequality \eqref{gron} and \eqref{ET5}, we obtain
$$
\|\sE^{(1)}_T-\sE^{(2)}_T\|_{\delta}^2\lesssim \int^T_0\bE |\sE^{(2)}_t(\tilde b^{(1)}_t(Z_t)-\tilde b^{(2)}_t(Z_t))|^2\dif t.
$$
Since $(p,q)\in\sI_1$, one can choose $\gamma\in(1,1/(d/p+2/q))$ so that 
$$(p/(2\gamma),q/(2\gamma))\in\sI_2.$$
Thus by H\"older's inequality and Krylov's estimate \eqref{Kry2}, we further have
\begin{align*}
\|\sE^{(1)}_T-\sE^{(2)}_T\|_{\delta}^2&\stackrel{\eqref{ET50}}{\lesssim} 
\left(\int^T_0\bE |\tilde b^{(1)}_t(Z_t)-\tilde b^{(2)}_t(Z_t)|^{2\gamma}\dif t\right)^{\frac{1}{\gamma}}\\
&\lesssim \nor |\tilde b^{(1)}-\tilde b^{(2)}|^{2\gamma}\nor^{1/\gamma}_{\widetilde\mL^{p/(2\gamma)}_{q/(2\gamma)}(T)}
=\nor\tilde b^{(1)}-\tilde b^{(2)}\nor^2_{\widetilde\mL^p_q(T)}\\
&\lesssim \left(\int^T_0\nor b_t(\cdot,\mu^{(1)}_t)-b_t(\cdot,\mu^{(2)}_t)\nor^{q}_p\dif t\right)^{\frac{2}{q}}\\
&\stackrel{\eqref{CC1}}{\lesssim} \left(\int^T_0\ell_t^{q}\|\phi_\theta\cdot (\mu^{(1)}_t-\mu^{(2)}_t)\|^{q}_{TV}\dif t\right)^{\frac{2}{q}},
\end{align*}
which together with \eqref{LK9} yields 
$$
\|\phi_\theta\cdot (\mu^{(1)}_T-\mu^{(2)}_T)\|^q_{TV}\leq C\int^T_0\ell_t^{q}\|\phi_\theta\cdot (\mu^{(1)}_t-\mu^{(2)}_t)\|^{q}_{TV}\dif t.
$$
By Gronwall's inequality, we obtain
$$
\|\phi_\theta\cdot (\mu^{(1)}_T-\mu^{(2)}_T)\|^{q}_{TV}=0\Rightarrow\mu^{(1)}_T=\mu^{(2)}_T.
$$
The proof is thus complete.
\end{proof}

\section{Application to nonlinear Fokker-Planck equations}
In this section we present some applications to nonlinear Fokker-Planck equations.
First of all we recall the following superposition principle: one-to-one correspondence between DDSDE \eqref{SDE} and nonlinear Fokker-Planck equation \eqref{Non},
which was first proved in \cite{Ba-Ro, Ba-Ro1}, and is based on a result for linear Fokker-Planck equations due to Trevisan \cite{Tr} (see also \cite{Fi} for the special linear case
where the coefficients are bounded). We repeat the argument from \cite{Ba-Ro, Ba-Ro1} here.
\bt[Superposition principle]\label{Th23}
Let $\mu_t:\mR_+\to\cP(\mR^d)$ be a continuous curve such that for each $T>0$,
\begin{align}\label{ET1}
\int^T_0\!\!\!\int_{\mR^d}\Big(|(\sigma^{ik}_t\sigma^{jk}_t)(x,\mu_t)|+|b_t(x,\mu_t)|\Big)\mu_t(\dif x)\dif t<\infty.
\end{align}
Then $\mu_t$ solves the nonlinear Fokker-Planck equation \eqref{Non} in the distributional sense 
if and only if there exists a martingale solution $\mP\in\sM^{\sigma,b}_\nu$ 
to DDSDE \eqref{SDE} so that for each $t>0$,
$$
\mu_t=\mP\circ w^{-1}_t.
$$
In particular, if there is at most one element in $\cM^{\sigma,b}_\nu$ with time martingale $\mu_t:=\mu_{X_t}, t\geq 0$, satisfying \eqref{ET1},
then there is at most one solution to \eqref{Non} satisfying \eqref{ET1}.
\et
\begin{proof}
If $\mP\in \sM^{\sigma,b}_\nu$ and $\mu_t=\mP\circ w^{-1}_t$, 
then by \eqref{ET1} and It\^o's formula, it is easy to see that $\mu_t$ solves \eqref{Non}. Now we assume $\mu_t$ 
solves \eqref{Non}. Consider the following linear Fokker-Planck equation:
$$
\p_t\tilde\mu_t=(\sL_t^{\sigma^\mu})^*\tilde\mu_t+\div(b^\mu_t\cdot\tilde\mu_t),
$$
where $b^\mu_t(x):=b_t(x,\mu_t)$ and $\sigma^\mu_t(x):=\sigma_t(x,\mu_t)$. Since $\mu_t$ is a solution of the above linear Fokker-Planck equation,
by \cite[Theorem 2.5]{Tr}, there is a martingale solution $\mP\in\sM^{\sigma^\mu, b^\mu}_\nu$ so that
$$
\mu_t=\mP\circ w_t^{-1}.
$$
In particular, $\mP\in\sM^{\sigma, b}_\nu$. The last assertion is then obvious and thus the proof is complete.
\end{proof}

From the above superposition principle and our well-posedness results, we can obtain the following wellposedness result about the nonlinear Fokker-Planck equations.
\bt
In the situations of Theorems \ref{Th51} and \ref{Th43}, there is a unique continuous curve $\mu_t$ solving the nonlinear Fokker-Planck equation \eqref{Non}.
\et

Now we turn to the proof of Theorem \ref{Th11}.
\begin{proof}[Proof of Theorem \ref{Th11}]
The existence and uniqueness of solutions to the nonlinear FPE \eqref{FK} 
are consequences of Theorem \ref{Th43} and Theorem \ref{Th23}. We now aim to show the existence and smoothness of the density $\rho^X_t(y)$.
Let $\mu_t$ be the solution of the Fokker-Planck equation \eqref{FK}. We consider the following SDE:
\begin{align}\label{SDE2}
\dif X_t=b^\mu_t(X_t)\dif t+\sqrt{2}\dif W_t, \ X_0=\xi,
\end{align}
where $b^\mu_t(x):=\int_{\mR^d}b_t(x,y)\mu_t(\dif y)$. Since $b^\mu\in\widetilde\mL^p_q$, where $\frac{d}{p}+\frac{2}{q}<1$,
it is well known that the operator $\Delta+b^\mu\cdot\nabla$ admits a heat kernel $\rho_{b^\mu}(s,x;t,y)$ (see \cite[Theorems 1.1 and 1.3]{Ch-Hu-Xi-Zh}),
which is continuous in $(s,x;t,y)$ on $\{(s,x;t,y): 0\leq s<t<\infty, x,y\in\mR^d\}$
and satisfies the following two-sided estimate: For any $T>0$, there are constants $c_0,\gamma_0>1$ such that for all $0\leq s<t\leq T$ and $x,y\in\mR^d$
$$
c_0^{-1}(t-s)^{-d/2}\e^{-\gamma_0 |x-y|^2/(t-s)}\leq \rho_{b^\mu}(s,x;t,y)\leq c_0(t-s)^{-d/2}\e^{-|x-y|^2/(\gamma_0(t-s))}, 
$$
and the gradient estimate: for some $c_1,\gamma_1>1$,
$$
|\nabla_x\rho_{b^\mu}(s,x;t,y)|\leq c_1(t-s)^{-(d+1)/2}\e^{-|x-y|^2/(\gamma_1(t-s))}. 
$$
If $\div b\equiv 0$, then $\rho_{b^\mu}(s,x;t,y)=\rho_{-b^\mu}(s,y;t,x)$, and so in this case,
$$
|\nabla_y\rho_{b^\mu}(s,x;t,y)|\leq c_1(t-s)^{-(d+1)/2}\e^{-|x-y|^2/(\gamma_1(t-s))}. 
$$
In particular, the density of the law of $X_t$ is just given by
$$
\rho^X_t(y)=\int_{\mR^d}\rho(0,x;t,y)(\bP\circ X^{-1}_0)(\dif x).
$$
Strong uniqueness of SDE \eqref{SDE2} ensures that $\rho^X_t(y)\dif y=\mu_t(\dif y)$.
The desired estimates now follow from the above estimates.
\end{proof}

{\bf Acknowledgement:} The authors thank Dr. Xing Huang for pointing out an error in the earlier version.


\begin{thebibliography}{999}

\bibitem{Ba-Ro}Barbu V. and R\"ockner M.: 
Probabilistic representation for solutions to nonlinear Fokker–Planck equations. {\it SIAM J. Math. Anal. \bf 50} (2018), no. 4, 4246-4260.

\bibitem{Ba-Ro1}Barbu V. and R\"ockner M.: From nonlinear Fokker-Planck equations to solutions of distribution dependent SDE.
arXiv:1808.10706.

\bibitem{Bo-Kr-Ro-Sh}Bogachev V.I., Krylov N.V. R\"ockner M. and Shaposhinikov S.V.: {\it Fokker-Planck-Kolmogorov equations.} AMS, 2015.

\bibitem{Ca-De}Carmona R. and Delarue F.: Probabilistic analysis of mean-field games. {\it SIAM J. Control and Optimization}, 51(4):2705-2734, 2013.

\bibitem{Ca-De1}Carmona R. and Delarue F.: {\it Probabilistic theory of mean field games with applications. II. Mean field games with common noise and master 
equations.} Probability Theory and Stochastic Modeling, 84. Springer, 2018.

\bibitem{Ca-Gv-Pa-Sc}Carrillo J.A., Gvalani R.S., Pavliotis G.A. and Schlichting A.: Long-time behavior 
and phase transitions for the McKean-Vlasov equation on the torus. arXiv: 1806.01719v2.

\bibitem{Ch} Chiang T.: McKean-Vlasov equations with discontinuous coefficients. {\it Soochow J. Math.}, 20(4):507-526, 1994.

\bibitem{Ch-Hu-Xi-Zh}Chen Z.Q., Hu E., Xie L. and Zhang X.: Heat kernels for non-symmetric diffusion operators with jumps.
{\it J. Differential Equations} 263 (2017) 6576-6634.

\bibitem{Ha-Si-Sz}Hammersley W., Siska D. and Szpruch L.: Mckean-Vlasov SDEs under measure dependent Lyapunov conditions.
arXiv:1802.03974v2.

\bibitem{Hu-Wa}Huang X. and Wang F.Y.: Distribution dependent SDEs with singular coefficients. arXiv:1805.01682v1.

\bibitem{Fi}Figalli A.: Existence and uniqueness of martingale solutions for SDEs with rough or degenerate coefficients.
{\it J. Funct. Anal.} 254 (2008), no. 1, 109-153.

\bibitem{Fu}Funaki T.: A certain class of diffusion processes associalted with nonlinear parabolic equations. {\it Prob. Theory and Relat. Fields}, 67(3):331-348,1984.

\bibitem{Ja-Sh} Jacod J. and Shiryaev A. N.: {\it Limit theorems for stochastic processes}. Springer-Verlag, 2002.

\bibitem{Ka}Kac M.: Foundations of Kinetic Theory. In Proceedings of the Third Berkeley 
Symposium on Mathematical Statistics and Probability, Volume 3: Contributions to Astronomy and Physics,
pages 171–197, Berkeley, Calif., 1956. University of California Press.

\bibitem{Kr-Ro}Krylov N.V. and R\"ockner M.: Strong solutions of stochastic equations with singular time dependent drift.
{\it Probab. Theory Relat. Fields}, {\bf 131} (2005), 154-196.

\bibitem{Ku}Kurtz T.G.: Weak and strong solutions of general stochastic models. {\it Electron. Common. Probab.}, 19(58):1-16, 2014.

\bibitem{Li-Mi} Li  J. and Min H.: Weak solutions of mean-field stochastic differential equations and application to 
zero-sum stochastic differential games. {\it SIAM Journal on Control and Optimization}, 54(3):1826-1858, 2016.

\bibitem{Ma-Ro-Sh}Manita O.A., Romanov M.S., Shaposhnikov S.V.: On uniqueness of solutions to nonlinear Fokker-Planck-Kolmogorov equations.
 {\it Nonlinear Anal.} 128 (2015), 199-226.
 
\bibitem{Ma-Sh}Manita, O. A.; Shaposhnikov, S. V.: Nonlinear parabolic equations for measures. (Russian) Algebra i Analiz 25 (2013), no. 1, 64--93; 
translation in St. Petersburg Math. J. 25 (2014), no. 1, 43–62 .

\bibitem{Mc} McKean H. P.: A class of Markov processes associated with nonlinear parabolic equations. {\it Proc Nat. Acad Sci USA}, 56(6):1907-1911, 1966.

\bibitem{Mi-Ve}Mishura Y.S. and Veretennikov A.Y.: Existence and uniqueness theorems for 
solutions of McKean-Vlasov stochastic equations, arXiv:1603.02212v4.

\bibitem{Sc}Scheutzow M.: A stochastic Gronwall's lemma. {\it Infinite Dimensional Analysis, Quantum Probability 
and Related Topics}, {\bf 16}, No. 2 (2013) 1350019 (4 pages).

\bibitem{St-Va}Stroock D.W. and Varadhan S.S.: {\it Multidimensional diffusion processes}. Springer-Verlag, Berlin, 1979. 	

\bibitem{Sz}Sznitman A.S.: {\it Topics in propagation of chaos}. In \'Ecole d'\'Et\'e de Prob. de Saint-Flour XIX-1989, 
Vol. 1464, {\it Lect. Notes in Math.}, pages 165-251. Springer-Verlag, 1991.

\bibitem{Tr}Trevisan D.: Well-posedness of multidimensional diffusion processes with weakly differentiable coefficients. 
{\it Electron. J. Probab.\bf 21}, (2016), Paper No. 22, 41 pp.

\bibitem{Va}Vlasov A. A.: The vibrational properties of an electron gas. {\it Soviet Physics Uspekhi}, 10(6):721, 1968.

\bibitem{Vi}Villani C.: {\it Optimal transport: old and new}. Springer-Verlag, 2009.

\bibitem{Wa}Wang F.Y.: Distribution dependent SDEs for Landau type equations. {\it Stoch. Procc. Appl.}, 128, no. 2, 595-621(2018).

\bibitem{Xi-Xi-Zh-Zh}Xia P., Xie L., Zhang X. and Zhao G.: $L^q(L^p)$-theory of stochastic differential equations. Preprint.

\bibitem{Xi-Zh}Xie L. and Zhang X.: Ergodicity of stochastic differential equations with jumps and singular coefficients. To appear in 
{\it Annales de l'Institut Henri Poincar\'e - Probabilit\'es et Statistiques}, arXiv:1705.07402.

\bibitem{Zh0}Zhang X.: Stochastic homeomorphism flows of SDEs with singular drifts and Sobolev diffusion coefficients.
{\it Electron. J. Probab. \bf 16} (2011), 1096-1116.

\bibitem{Zh1}Zhang X.: A discretized version of Krylov's estimate and its applications. arXiv:1909.09976.

\bibitem{Zh-Zh}Zhang X. and Zhao G.: Heat kernel and ergodicity of SDEs with distributional drifts.  arXiv.1710.10537.

\bibitem{Zh-Zh1}Zhang X. and Zhao G.: Singular Brownian Diffusion Processes. {\it Communications in Mathematics and Statistics,} pp.1-49, 2018. 

\bibitem{Zh-Zh2}Zhang X. and Zhao, G.: Stochastic Lagrangian path for Leray solutions of 3D Navier-Stokes
  equations,   arXiv: 1904.04387.

\bibitem{Zv}Zvonkin, A.K.: A transformation of the phase space of a diffusion process
that removes the drift.  {\it Mat. Sbornik}. {\bf 93} (135) (1974), 129-149.

\end{thebibliography}
\end{document}